\documentclass[11pt,a4paper]{article}

\usepackage[T1]{fontenc}
\usepackage[utf8]{inputenc}
\usepackage{amsmath,amssymb,amsthm}
\usepackage[margin=2.7cm]{geometry}
\usepackage[colorlinks=true,linkcolor=blue,citecolor=blue,urlcolor=blue]{hyperref}

\newtheorem{theorem}{Theorem}
\newtheorem{lemma}{Lemma}
\newtheorem{proposition}{Proposition}
\theoremstyle{remark}
\newtheorem{remark}{Remark}

\newcommand{\R}{\mathbb{R}}
\newcommand{\de}{\mathrm{d}}
\newcommand{\E}{\mathbb{E}}
\newcommand{\PP}{\mathbb{P}}
\newcommand{\F}{\mathcal{F}}
\newcommand{\ind}[1]{\mathbf{1}_{\{#1\}}}
\newcommand{\ptil}{\widetilde{p}}
\newcommand{\ptileps}{\widetilde{p}_{\epsilon}}
\newcommand{\Htil}{\overline{H}}
\newcommand{\mtil}{\widetilde{m}}
\newcommand{\Deln}{\Delta^{(n)}}
\newcommand{\Delnm}{\widetilde{\Delta}^{(n,m)}}
\newcommand{\Tn}{\mathbb{T}_n}
\newcommand{\Tm}{\mathcal{T}_m}
\DeclareMathOperator{\Tr}{Tr}

\title{Strong invariance principles for diffusions, Markov chains and their perturbations\thanks{The article was prepared for the first and second authors within the framework of the Basic Research Program at HSE University.}}
\author{V.~Konakov\thanks{National Research University Higher School of Economics (HSE), Russian Federation.}
\and D.~Kucher\thanks{National Research University Higher School of Economics (HSE), Moscow, Russia.}
\and E.~Mammen\thanks{Institute for Mathematics, Heidelberg University, Germany, mammen@math.uni-heidelberg.de}}
\date{}

\begin{document}
\maketitle



\begin{abstract}
In this paper, we construct strong approximations for discrete-time Markov chains weakly converging to continuous diffusion processes, as well as for their perturbed counterparts. Under the assumption of bounded coefficients, we construct closely coupled versions of these processes on  shared probability spaces. In particular for both, non-degenerate and degenerate cases, we maximize the probability of their exact pathwise coincidence on discrete time grids. Moreover, we construct couplings such that the probability of small deviations between the interpolated Markov chain and the continuous diffusion trajectory is small on the entire time interval. 
\end{abstract}

\section{Introduction}

For a fixed given deterministic final horizon $T>0$ let us consider the following multidimensional SDE
\begin{equation}\label{eq:sde}
\de X_t = b(t,X_t)\,\de t + \sigma(t,X_t)\,\de W_t, \quad t\in[0,T],
\end{equation}
where $b:[0,T]\times\R^d\to\R^d$, $\sigma:[0,T]\times\R^d\to\R^d\otimes\R^d$ are bounded coefficients that are measurable in time and H\"older continuous in space and $W$ is a Brownian motion on some filtered probability space $(\Omega,\F,(\F_t)_{t\ge0},\PP)$. Also, $a(t,x):=\sigma\sigma^*(t,x)$ is assumed to be uniformly elliptic. In particular those assumptions guarantee that \eqref{eq:sde} admits a unique weak solution, see e.g.\ \cite{bassperkins2009,menozzi2011}, from which the uniqueness to the martingale problem for the associated generator can be derived.

We now introduce, for a given parameter $\epsilon>0$, a perturbed version of \eqref{eq:sde} with dynamics
\begin{equation}\label{eq:sdepert}
\de X_t^{\epsilon} = b_{\epsilon}(t,X_t^{\epsilon})\,\de t + \sigma_{\epsilon}(t,X_t^{\epsilon})\,\de W_t, \quad t\in[0,T],
\end{equation}
where $b_{\epsilon}:[0,T]\times\R^d\to\R^d$, $\sigma_{\epsilon}:[0,T]\times\R^d\to\R^d\otimes\R^d$ satisfy similar assumptions as $b,\sigma$ and where they are meant to be close to $b,\sigma$ when $\epsilon$ is small. We assume that  $T$ is fixed. The case of a growing time horizon $T\nearrow\infty$ is discussed in Remark~\ref{rem:largeT} at the end of the paper.

The goal of this work is  to construct close versions of the original process $X_t$ and the perturbed process $X_t^{\epsilon}$ on a common probability space. Important applications of these couplings can for instance be found in mathematical finance \cite{corielli2010}, \cite{benhamou2010}. If the dynamics of \eqref{eq:sde} models the evolution of the (log-)price of a financial asset, it is often very useful to know how a perturbation of the volatility $\sigma$ impacts the dynamics of the processes, and therefore the associated option prices. Another important application includes the case of mollification by spatial convolution. In this framework, some explicit convergence results can be found in \cite{konakovmenozzi2017}. 
We expect further applications to parameter estimation in diffusion models \cite{genoncatalot1993} and to Euler approximation schemes \cite{konakovmammen2002,bally1996,konakovmenozzi2017,mikulevicius1991,hamaguchi2024,babi2022}. More recent results state stability estimates in  comparing two distinct solutions $X^1,X^2$ driven by the same $W$ but with different initial conditions, drifts and/or diffusion coefficients. Here it is worth to mention \cite{bogachev2016,bitter2021,galeati2023}. Another type of perturbations is associated with the replacement of small jumps part by Brownian diffusion in the SDE driven by a compensated Poisson random measure \cite{bally2024}.

Additionally to the diffusion model \eqref{eq:sde} we consider Kolmogorov type degenerate diffusions. 
Stability results were shown for these processes in \cite{konakovmenozzimolchanov2010}, \cite{kozhina2017}. In the degenerate case, the diffusion equation takes the form:
\begin{equation}\label{eq:degsde}
\begin{cases}
\de \widetilde X_t = b(\widetilde X_t,\widetilde Y_t)\,\de t + \sigma(\widetilde X_t,\widetilde Y_t)\,\de W_t, \\
\de \widetilde Y_t = \widetilde X_t\,\de t, \qquad t\in[0,T],
\end{cases}
\end{equation}
where $W_t$ is again a standard $d$-dimensional Brownian motion and where $b:\R^{2d}\to\R^d$, $\sigma:\R^{2d}\to\R^d\otimes\R^d$ are bounded H\"older continuous coefficients. This system is classified as degenerate because the total diffusion matrix for the full vector of coefficients is singular. To see this, we rewrite the system in a unified vector form by defining the full vector $\widetilde Z_t=(\widetilde X_t,\widetilde Y_t)^{\top}\in\R^{2d}$:
\begin{equation*}
\de \widetilde Z_t = B(\widetilde Z_t)\,\de t + \Sigma(\widetilde Z_t)\,\de W_t,
\end{equation*}
where the drift vector $B(\widetilde Z_t)$ and the matrix $\Sigma(\widetilde Z_t)\in\R^{2d\times d}$ are given by:
\begin{equation*}
B(\widetilde Z_t) = \begin{pmatrix} b(\widetilde X_t,\widetilde Y_t) \\ \widetilde X_t \end{pmatrix}, \qquad
\Sigma(\widetilde Z_t) = \begin{pmatrix} \sigma(\widetilde X_t,\widetilde Y_t) \\ 0_{d\times d} \end{pmatrix}.
\end{equation*}
The infinitesimal generator associated with this diffusion process depends on the $2d\times2d$ diffusion matrix $A(\widetilde Z_t)$, which is constructed as:
\begin{equation*}
A(\widetilde Z_t) = \Sigma(\widetilde Z_t)\Sigma(\widetilde Z_t)^{\top} =
\begin{pmatrix} \sigma(\widetilde X_t,\widetilde Y_t)\sigma(\widetilde X_t,\widetilde Y_t)^{\top} & 0_{d\times d} \\ 0_{d\times d} & 0_{d\times d} \end{pmatrix}.
\end{equation*}
We will construct a coupling of the processes $\widetilde Z_t=(\widetilde X_t,\widetilde Y_t)^{\top}\in\R^{2d}$ and $\widetilde Z^{\epsilon} _t=(\widetilde X^{\epsilon} _t,\widetilde Y^{\epsilon} _t)^{\top}\in\R^{2d}$:
\begin{equation}\label{eq:degsdepert}
\begin{cases}
\de \widetilde X^{\epsilon} _t = b_{\epsilon} (\widetilde X^{\epsilon} _t,\widetilde Y^{\epsilon} _t)\,\de t + \sigma_{\epsilon}(\widetilde X^{\epsilon} _t,\widetilde Y^{\epsilon} _t)\,\de W_t, \\
\de \widetilde Y^{\epsilon} _t = \widetilde X^{\epsilon} _t\,\de t, \qquad t\in[0,T],
\end{cases}
\end{equation}
where again $b_{\epsilon}:\R^{2d}\to\R^d$, $\sigma_{\epsilon}:\R^{2d}\to\R^d\otimes\R^d$ satisfy similar assumptions as $b,\sigma$.

Furthermore we will consider discrete-time Markov chains defined on a time grid $0=t_0<t_1<\dots<t_n=T$ with step size $h=t_{i+1}-t_i$. The Markov chain associated with \eqref{eq:sde} is given by:
\begin{equation}\label{eq:chain}
Y_{t_{i+1}} = Y_{t_i} + b(t_i,Y_{t_i})h + \sigma(t_i,Y_{t_i})\sqrt{h}\,\xi^n_{i+1},
\end{equation}
where $\{\xi^n_i\}_{i=1}^n$ are centered random variables. We also consider a perturbed version of \eqref{eq:chain}:
\begin{equation}\label{eq:chainpert}
Y^{\epsilon}_{t_{i+1}} = Y^{\epsilon}_{t_i} + b_{\epsilon}(t_i,Y^{\epsilon}_{t_i})h + \sigma_{\epsilon}(t_i,Y^{\epsilon}_{t_i})\sqrt{h}\,\xi^n_{i+1}.
\end{equation}
We call $\Tn=\{t_0,t_1,\dots,t_n\}=\{0,T/n,\dots,T\}$ the micro-grid. Below we also consider the macro-grid $\Tm=\{\tau_0,\tau_1,\dots,\tau_m\}=\{0,T/m,\dots,T\}$, where $m$ divides $n$, so that $\Tm\subset\Tn$.

Besides the perturbed models, we will establish strong approximations of Markov chains $Y^{\epsilon}$ by their diffusion limit $X$. Here we compare the Markov chain with perturbed coefficients $b_{\epsilon}$ and $\sigma_{\epsilon}$ to allow for different coefficients for the Markov chains and the approximating diffusion. In this comparison we will make Markov assumptions for the errors $\xi^n_i$. More precisely, we will assume that the conditional distribution of $\xi^n_{i+1}$ given $Y^{\epsilon}_{t_i}=y_i, Y^{\epsilon}_{t_{i-1}}=y_{i-1}, ...$ only depends on $y_i$:
\begin{equation}\label{eq:markov}
\mathcal{L}\left(\xi^n_{i+1}\mid Y^{\epsilon}_{t_i}=y_i,\,Y^{\epsilon}_{t_{i-1}}=y_{i-1},\dots\right) = q^n_{t_i,y_i}(\cdot).
\end{equation}
In our comparison of the process $Y_{t_i}$ and its perturbed counter part $Y^{\epsilon}_{t_i}$ we will make the stronger assumption that the errors $\xi^n_i$ are independent. We will also make use of linear interpolations of the Markov chain which is defined as follows. 
\begin{equation*}
Y(t) = \sum_{k=0}^{n-1}\left(\frac{t_{k+1}-t}{t_{k+1}-t_k}\,Y_{t_k} + \frac{t-t_k}{t_{k+1}-t_k}\,Y_{t_{k+1}}\right)\mathbf{1}_{[t_k,t_{k+1})}(t) + Y_{t_n}\mathbf{1}_{\{t_n\}}(t).
\end{equation*}

Our paper is organized as follows. In the next section \ref{sec:assumtions} we will state our assumptions. Section \ref{sec:results} contains our main results.

\section{Assumptions} \label{sec:assumtions}

\paragraph{Notation.} For $a,b>0$ let
\begin{equation}\label{eq:MLdef}
\mathcal{E}_{a,b}(z):=\sum_{r\ge0}\frac{z^r}{\Gamma(ar+b)}
\end{equation}
be the Mittag-Leffler function. It is entire, positive and increasing on $[0,\infty)$, and for $z\ge0$
\begin{equation}\label{eq:ML}
\sum_{r\ge1}\frac{z^r}{\Gamma(ar+b)}=z\,\mathcal{E}_{a,a+b}(z),\qquad
\sum_{r\ge0}\frac{(r+1)\,z^r}{\Gamma(ar+b)}\le\mathcal{E}_{a,b}(2z),
\end{equation}
the second inequality because $r+1\le2^r$. For $0<a<2$ one has the asymptotics $\mathcal{E}_{a,b}(z)=a^{-1}z^{(1-b)/a}\exp\left(z^{1/a}\right)(1+o(1))$ as $z\to+\infty$. The Mittag-Leffler functions describe the dependence of all constants below on the horizon $T$.

For $r>d$ we denote by $Q_r$ the rational function
\begin{equation*}
Q_r(z) := \frac{c_r}{(1+|z|)^r},\qquad z\in\R^d,
\end{equation*}
where the constant $c_r$ is chosen such that $\int_{\R^d} Q_r(z)\,dz = 1$.

\medskip
\noindent\textbf{(A1) (Bounded coefficients).} The coefficients are measurable and bounded: there exist constants $K_1,K_2>0$ such that
\begin{align*}
&\sup_{t\in[0,T],\,x\in\R^d} \big(|b(t,x)| + |\sigma(t,x)|\big) \le K_1, \\
&\sup_{t\in[0,T],\,x\in\R^d} \big(|b_{\epsilon}(t,x)| + |\sigma_{\epsilon}(t,x)|\big) \le K_2.
\end{align*}
Thus $K_1$ bounds the coefficients $b,\sigma$ of the original processes (the diffusion \eqref{eq:sde} and the chain \eqref{eq:chain}), and $K_2$ bounds the coefficients $b_{\epsilon},\sigma_{\epsilon}$ of the perturbed processes (the diffusion \eqref{eq:sdepert} and the chain \eqref{eq:chainpert}); in the degenerate case \eqref{eq:degsde}, \eqref{eq:degsdepert} the coefficients are functions of $(x,y)\in\R^{2d}$ and the supremum is taken over $(x,y)\in\R^{2d}$. Here $|\cdot|$ denotes the Euclidean norm of vectors and the operator norm of matrices. Only the common bound
\begin{equation*}
K:=K_1\vee K_2
\end{equation*}
enters the results below.

\medskip
\noindent\textbf{(A2) (Uniform Ellipticity).} The covariance matrices $a:=\sigma\sigma^*$ and $a_{\epsilon}:=\sigma_{\epsilon}\sigma_{\epsilon}^*$ are uniformly elliptic, i.e., there exists a constant $\lambda\ge1$ such that for all $(x,y,\xi)\in(\R^d)^3$:
\begin{align*}
\lambda^{-1}|\xi|^2 &\le \langle a(x,y)\xi,\xi\rangle \le \lambda|\xi|^2, \\
\lambda^{-1}|\xi|^2 &\le \langle a_{\epsilon}(x,y)\xi,\xi\rangle \le \lambda|\xi|^2.
\end{align*}

\medskip
\noindent\textbf{(A3) (H\"older and Lipschitz continuity in space).} Coefficients $b,b_{\epsilon},\sigma,\sigma_{\epsilon}$ satisfy the following conditions: for all pairs $(t,x),(s,y)\in[0,T]\times\R^d$, there exist constants $A,B,\alpha_2,\beta,\gamma_2>0$ such that:
\begin{align*}
|\sigma(t,x)-\sigma(s,y)| + |\sigma_{\epsilon}(t,x)-\sigma_{\epsilon}(s,y)| &\le A\left(|x-y|^{\gamma_2}+|t-s|^{\alpha_2}\right), \\
|b(t,x)-b(s,y)| + |b_{\epsilon}(t,x)-b_{\epsilon}(s,y)| &\le B\left(|x-y|+|t-s|^{\beta}\right), \qquad \gamma_2,\beta,\alpha_2\le1.
\end{align*}
Note that, since $\sigma$ and $\sigma_{\epsilon}$ are bounded, the covariance matrices $a$ and $a_{\epsilon}$ are then also $\gamma_2$-H\"older in space and $\alpha_2$-H\"older in time, with constants depending on $A$ and $K_1,K_2$.

\medskip
\noindent\textbf{(AD3) (H\"older continuity in space for the degenerate case).} In the degenerate case there exist positive constants $C_D>0$ and $0<\gamma_3\le1$ such that:
\begin{equation*}
|b(x,y)-b(x',y')| + |\sigma(x,y)-\sigma(x',y')| \le C_D\left(|x-x'|^{\gamma_3}+|y-y'|^{\gamma_3/3}\right),
\end{equation*}
for all $(x,y),(x',y')\in\R^d\times\R^d$.

\medskip
\noindent\textbf{(A4) (Regularity of the density family).} The family of densities $q^n_{t,x}$ of the errors $\xi^n_i$ in \eqref{eq:markov} decays polynomially at infinity, uniformly in $n$ and $\epsilon$, along with its derivatives. Furthermore, it satisfies the following conditions with some constant $C_q>0$ independent of $n$:
\begin{align*}
|D_z^{\nu} q^n_{t,x}(z)| &\le C_q Q_S(z) \qquad \forall (t,x)\in[0,T]\times\R^d,\ z\in\R^d,\ \text{uniformly in } n, \\
|q^n_{t,x}(z)-q^n_{t,y}(z)| &\le C_q\sqrt{h}\,|x-y|\,Q_S(z) \qquad \forall (x,y,z)\in(\R^d)^3,
\end{align*}
for some $S>2d+6$ and $|\nu|\le5$. Here $h=T/n$ is the step of the micro-grid; for $T=1$ the second condition reads $|q^n_{t,x}(z)-q^n_{t,y}(z)|\le C_qn^{-1/2}|x-y|Q_S(z)$.

\medskip
\noindent\textbf{(A5) (Independent errors).} In the comparison of the chain \eqref{eq:chain} with its perturbation \eqref{eq:chainpert} the errors $(\xi^n_k)_{k\ge1}$ are independent centered random variables with a $C^5$ density $f^n_{\xi}$. For a fixed integer $M>2d+5+\gamma_2$, this density and its partial derivatives up to order 5 decay polynomially of order $M$, uniformly in $n$: for all $z\in\R^d$ and all multi-indices $\nu$ with $|\nu|\le5$,
\begin{equation*}
|D^{\nu}f^n_{\xi}(z)| \le C_{\xi}\, Q_M(z),
\end{equation*}
with a constant $C_{\xi}$ independent of $n$.

Assumption \textbf{A4} is used only for the comparison of the chain with the diffusion (Markov errors, Theorem~\ref{thm:chaindif}), and \textbf{A5} only for the comparison of the chain with its perturbation (independent errors, Theorem~\ref{thm:chain}).

Below we will make use of the following distances between the original and perturbed coefficients:
\begin{align*}
\Delta_{\epsilon,\sigma,\gamma} &= \sup_{t\in[0,T]} |\sigma(t,\cdot)-\sigma_{\epsilon}(t,\cdot)|_{\gamma}, \\
\Delta_{\epsilon,b,q} &= \sup_{t\in[0,T]} \|b(t,\cdot)-b_{\epsilon}(t,\cdot)\|_{L^q(\R^d)},
\end{align*}
where $|\cdot|_{\gamma}$ denotes the H\"older norm and $\|\cdot\|_{L^q}$ is the $L^q$ norm. We define the total perturbation parameter as:
\begin{equation*}
\Delta_{\epsilon,\gamma,q} := \Delta_{\epsilon,\sigma,\gamma} + \Delta_{\epsilon,b,q}.
\end{equation*}

\section{Main results} \label{sec:results}
\subsection{Uniform strong approximations}\label{sec:results1} In this subsection we state our main results. For the diffusion process, the degenerate diffusion and the Markov process we state couplings of the processes with their perturbed counterparts. Theorem \ref{thm:dif} treats the diffusion process, Theorem \ref{thm:degdif} the degenerate diffusion and Theorem  \ref{thm:chain} the Markov process. Furthermore, in Theorem \ref{thm:chaindif} we show a strong approximation result for a Markov process and its approximating diffusion process. All approximations are shown to hold uniformly over the full interval $[0,T]$. 

In Subsection \ref{sec:results4} we will show that the results of this subsection directly follow from the results stated in Subsections \ref{sec:results2} and \ref{sec:results3}. Subsection \ref{sec:results2} contains results on strong approximations of the processes on a finite grid of points contained in $[0,T]$. We will show that there exist couplings of the processes compared with their perturbed counterparts and of the Markov chain compared with its approximating diffusion such that in the coupling the pairs of processes coincide on the grid with high probability. In Subsection \ref{sec:results3} we will state bounds for local fluctuations of the processes that hold with high probability.

\paragraph{Notation and constants.} Throughout, $K=K_1\vee K_2$ is the bound of the coefficients from \textbf{A1}, $\frac12<\alpha<1$, $\omega_d:=2\pi^{d/2}/\Gamma(d/2)$ is the area of the unit sphere in $\R^d$, and $\mathcal{E}_{a,b}$ is the Mittag-Leffler function \eqref{eq:MLdef}. The results involve four constants $C_1,\dots,C_4$, numbered after the theorems in which they appear. They are explicit functions of $T$; the constants $c_1,c_2,c_3\ge1$ entering them depend only on $d$, $q$ and the constants in the assumptions, but not on $T$. Their derivation from \cite{konakov2017stability} and \cite{kozhina2018} is given at the beginning of the proof of Propositions~\ref{thm:pert} and~\ref{thm:deg}.

For $q\in(d,\infty]$,
\begin{equation}\label{eq:CKKM}
\begin{gathered}
C_1(q,T):=c_1\max(1,T)^{a_q}\;\mathcal{E}_{\vartheta_q,1}\!\Big(2c_1\max(1,T)^{a_q}\,\Gamma(\vartheta_q)\,T^{\vartheta_q}\Big),\\
\vartheta_q:=\min\Big(\frac{\gamma_2}{2},\frac{q-d}{2q}\Big),\qquad a_q:=\frac{1-\gamma_2}{2}+\Big|\frac{\gamma_2}{2}-\frac{q-d}{2q}\Big|\ \ (q<\infty),\qquad a_{\infty}:=\frac{1-\gamma_2}{2},
\end{gathered}
\end{equation}
and $C_3(q,T)$ is defined by the same formula with $c_3$ in place of $c_1$. We write $C_3(T):=C_3(\infty,T)$, i.e.\
\begin{equation}\label{eq:C3}
C_3(T)=c_3\max(1,T)^{\frac{1-\gamma_2}{2}}\;\mathcal{E}_{\frac{\gamma_2}{2},1}\!\Big(2c_3\max(1,T)^{\frac{1-\gamma_2}{2}}\,\Gamma\big(\tfrac{\gamma_2}{2}\big)\,T^{\frac{\gamma_2}{2}}\Big).
\end{equation}
For $q\in(4d,\infty]$,
\begin{equation}\label{eq:Cdeg}
\begin{gathered}
C_2(q,T):=c_2\max(1,T)^{b_q}\big(1+2T^{2\vartheta^K_q}\big)\;\mathcal{E}_{\vartheta^K_q,1}\!\Big(2c_2\max(1,T)^{b_q}\,T^{\vartheta^K_q}\Big),\\
\vartheta^K_q:=\min\Big(\frac{\gamma_3}{2},\frac12-\frac{2d}{q}\Big),\qquad b_q:=\frac{1-\gamma_3}{2}+\Big|\frac{\gamma_3}{2}-\frac12+\frac{2d}{q}\Big|\ \ (q<\infty),\qquad b_{\infty}:=\frac{1-\gamma_3}{2}.
\end{gathered}
\end{equation}
Finally, $C_4(T)$ is the constant of the density comparison of Proposition~\ref{thm:density}, given explicitly in \eqref{eq:finalT}.

\begin{theorem}\label{thm:dif}
Assume \textbf{A1}, \textbf{A2}, \textbf{A3}. Then for every $m\in\mathbb{N}$, every $\frac12<\alpha<1$ and every $q>d$ there exists a probability space supporting versions of the diffusion processes $X_{t}$ and $X^{(\epsilon)}_{t}$ defined in \eqref{eq:sde} and \eqref{eq:sdepert} such that
\begin{equation*}
\PP\left(\sup_{s\in[0,T]} |X_s-X_s^{\epsilon}| > 2m^{\alpha-1}+\frac{2KT}{m}\right) \le 4dm\,\exp\left(-\frac{m^{2\alpha-1}}{2dK^2T}\right) + C_1(q,T)\,m\,\Delta_{\epsilon,\gamma,q}.
\end{equation*}
\end{theorem}

\begin{theorem}\label{thm:degdif}
Assume \textbf{A1}, \textbf{A2}, \textbf{AD3} hold. Then for every $m\in\mathbb{N}$, every $\frac12<\alpha<1$ and every $q\in(4d,\infty]$ there exists a probability space supporting versions of the processes $\widetilde Z_{t}=(\widetilde{X}_{t},\widetilde{Y}_{t})$ and $\widetilde Z^{\epsilon}_{t}=(\widetilde{X}^{(\epsilon)}_{t},\widetilde{Y}^{(\epsilon)}_{t})$ such that
\begin{equation*}
\PP\left(\sup_{s\in[0,T]} |\widetilde X_s-\widetilde X_s^{\epsilon}| > 2m^{\alpha-1}+\frac{2KT}{m}\right) \le 4dm\,\exp\left(-\frac{m^{2\alpha-1}}{2dK^2T}\right) +C_2(q,T)\,m\,\overline{\Delta}_{\epsilon,\gamma,q}.
\end{equation*}
Here the distance $\overline{\Delta}_{\epsilon,\gamma,q}$ between the original and perturbed coefficients is defined as follows:
\begin{equation*}
\overline{\Delta}_{\epsilon,\gamma,q} := \|\sigma-\sigma^{\epsilon}\|_{d,\gamma} + \|b-b^{\epsilon}\|_{L^q(\R^{2d})},
\end{equation*}
where $q\in(1,+\infty]$, $\gamma\in(0,1]$, and the norms take the following forms on the space $\R^{2d}$:
\begin{align*}
\|\sigma-\sigma^{\epsilon}\|_{d,\gamma} &:= \sup_{z\in\R^{2d}} |\sigma(z)-\sigma^{\epsilon}(z)| + \sup_{z\ne z'\in\R^{2d}} \frac{|(\sigma-\sigma^{\epsilon})(z)-(\sigma-\sigma^{\epsilon})(z')|}{\left(|x-x'|+|y'-y|^{1/3}\right)^{\gamma}}, \\
\|b-b^{\epsilon}\|_{L^q(\R^{2d})} &:= \left(\int_{\R^{2d}} |b(z)-b^{\epsilon}(z)|^q\,dz\right)^{1/q}.
\end{align*}
\end{theorem}

\begin{theorem}\label{thm:chain}
Assume \textbf{A1}, \textbf{A2}, \textbf{A3} and \textbf{A5}. Then for every $m$ dividing $n$, every $R>0$ and every $\frac12<\alpha<1$ there exists a probability space supporting versions of the Markov processes $Y_{t}$ and $Y^{(\epsilon)}_{t}$ defined in \eqref{eq:chain} and \eqref{eq:chainpert} (linearly interpolated) such that
\begin{equation*}
\begin{aligned}
&\PP\left(\sup_{s\in[0,T]} |Y_s-Y_s^{\epsilon}| \ge 2m^{\alpha-1}+\frac{2KT}{m}+\frac{2c\sqrt{Tn}}{m\,R^{M-d-1}}\right)\\
&\qquad\le 4dm\,\exp\left(-\frac{m^{2\alpha-1}}{2dR^2T}\right) + \frac{2c\,n}{R^{M-d}} +C_3(T)\,m\,\Delta_{\epsilon,\gamma,\infty},
\qquad c:=\frac{C_{\xi}\,c_M\,\omega_d\,K^{M-d}}{M-d-1}.
\end{aligned}
\end{equation*}
If the errors are Gaussian, $C_3(T)\,\Delta_{\epsilon,\gamma,\infty}$ may be replaced by $C_3(q,T)\,\Delta_{\epsilon,\gamma,q}$ for any $q>d$.
\end{theorem}

For the case of a Markov chain approximating a diffusion, let $n$ be the number of steps in the micro-grid.

\begin{theorem}\label{thm:chaindif}
Assume \textbf{A1}, \textbf{A2}, \textbf{A3}, let the errors of \eqref{eq:chainpert} satisfy the Markov condition \eqref{eq:markov}, and let $n\ge T$. Assume moreover that the conditional densities $q^{\epsilon,n}_{t_i,x}$ of the noise $\sigma_{\epsilon}(t_i,Y^{\epsilon}_{t_i})\xi^n_{i+1}$ given $Y^{\epsilon}_{t_i}=x$ satisfy \textbf{A4} (with $q^{\epsilon,n}$ in place of $q^n$), and denote by $a^{\epsilon}(t,x)$ their covariance matrices; $a^{\epsilon}$ and $b_{\epsilon}$ are assumed to satisfy \textbf{A2}, \textbf{A3}. Then for every $m$ dividing $n$, every $R>0$ and every $\frac12<\alpha<1$ there exists a probability space supporting versions of the processes $X_{t}$ and $Y^{(\epsilon)}_{t}$ defined in \eqref{eq:sde} and \eqref{eq:chainpert} (the chain linearly interpolated) such that
\begin{equation*}
\begin{aligned}
&\PP\left(\sup_{s\in[0,T]} |X_s-Y_s^{\epsilon}| \ge 2m^{\alpha-1}+\frac{2KT}{m}+\frac{c\sqrt{Tn}}{m\,R^{S-d-1}}\right)\\
&\qquad\le 2dm\,\exp\left(-\frac{m^{2\alpha-1}}{2dK^2T}\right)+2dm\,\exp\left(-\frac{m^{2\alpha-1}}{2dR^2T}\right) + \frac{c\,n}{R^{S-d}}\\
&\qquad\quad +C_4(T)\,m\left[\ln\Big(\frac{en}{m}\Big)\Big(\frac{T}{m}\Big)^{\frac{\gamma_2-1}{2}}\Big(\Big(\frac{T}{n}\Big)^{\min(\frac{\gamma_2}{2},\alpha_2,\beta)}+\Delta_a+\Delta_b\Big)+\Big(\frac{m}{n}\Big)^{\frac12}\right],
\end{aligned}
\end{equation*}
where $c:=C_q\,c_S\,\omega_d/(S-d-1)$, $\Delta_a:=\sup_{t,x}|a(t,x)-a^{\epsilon}(t,x)|$ and $\Delta_b:=\sup_{t,x}|b(t,x)-b_{\epsilon}(t,x)|$.
\end{theorem}

\subsection{Strong approximations on a grid} \label{sec:results2}

In this subsection we state that there exist couplings of the pairs of processes on a finite grid of points contained in $[0,T]$. In the coupling the pairs of processes coincide on the grid with high probability. The results of this subsection will be our main argument to get the results of Subsection \ref{sec:results1}. But we think that the results of this section are of independent interest.  The basic idea for the proofs of the results of this subsection is the following coupling approach. First we use bounds between the transition densities of the considered two processes. The transition densities give the joint densities on the grid for both processes. Let us denote the densities by $f_1$ and $f_2$. Assume that the L$_1$-distance between the densities is equal to  $\delta$.  We now argue that there exists a coupling $(V_1,V_2)$ with $V_1$ distributed according to $f_1$, with $V_2$ distributed according to $f_2$ and with $\PP(V_1=V_2)= 1-\delta/2$. First note that the L$_1$-norms of $h=$min$(f_1,f_2)$, the positive part $h_1$ of $f_1-f_2$ and of the positive part $h_2$ of $f_2-f_1$ are equal to $1-\delta/2$, $\delta/2$ and $\delta/2$, respectively. Now generate $U$ with distribution $h/(1-\delta/2)$,  $U_1$ with distribution $h_1/(\delta/2)$, and $U_2$ with distribution $h_2/(\delta/2)$. Finally, put $(V_1,V_2)$ equal to $(U,U)$ with probability $1-\delta/2$ and equal to $(U_1,U_2)$ with probability $\delta/2$. This gives that $\PP(V_1=V_2)= 1-\delta/2$. See also \cite{denhollander2012}. The essential remaining step to get the results of the following propositions is the application of L$_1$-distances  between transition densities of the two processes, respectively considered in each proposition.
\begin{proposition}\label{thm:pert}
When \textbf{A1}, \textbf{A2}, \textbf{A3} hold, there exists a probability space supporting versions of the processes $Z_{t}$ and $Z^{(\epsilon)}_{t}$, where $Z$ and $Z^{\epsilon}$ are equal to the diffusion $X_t$ and perturbed diffusion $X_t^{\epsilon}$ (and then $q\in(d,\infty]$ is arbitrary), or equal to the Markov chain $Y_t$ and perturbed Markov chain $Y_t^{\epsilon}$ with independent errors satisfying \textbf{A5} (and then $q=\infty$; any $q\in(d,\infty]$ if the errors are Gaussian), such that the probability of their paths coinciding on the entire grid satisfies:
\begin{equation*}
\PP\left(Z_{\tau_k}=Z^{(\epsilon)}_{\tau_k} \ \text{for all } k=1,\dots,m\right) \ge 1 - C_1(q,T)\,m\,\Delta_{\epsilon,\gamma,q}\qquad\text{(diffusions)},
\end{equation*}
respectively $\ge1-C_3(T)\,m\,\Delta_{\epsilon,\gamma,\infty}$ (chains; $\ge1-C_3(q,T)\,m\,\Delta_{\epsilon,\gamma,q}$ for Gaussian errors).
\end{proposition}

\begin{proposition}\label{thm:deg}
In the degenerate case where \textbf{A1}, \textbf{A2}, \textbf{AD3} hold, for every $q\in(4d,\infty]$ there exists a probability space supporting versions of the processes $V_{\tau_k}=(\widetilde{X}_{\tau_k},\widetilde{Y}_{\tau_k})$ and $V^{\epsilon}_{\tau_k}=(\widetilde{X}^{(\epsilon)}_{\tau_k},\widetilde{Y}^{(\epsilon)}_{\tau_k})$ such that the probability of their paths coinciding on the entire grid satisfies:
\begin{equation*}
\PP\left(V^{\epsilon}_{\tau_k}=V_{\tau_k} \ \text{for all } k=1,\dots,m\right) \ge 1 - C_2(q,T)\,m\,\overline{\Delta}_{\epsilon,\gamma,q}.
\end{equation*}
\end{proposition}

For the case of a Markov chain approximating a diffusion, let $n$ be the number of steps in the micro-grid, and let $0<\tau_1<\dots<\tau_m=T$ be a macroscopic grid with equal spacing $T/m$.

\begin{proposition}\label{thm:grid}
Assume that conditions \textbf{A1}, \textbf{A2}, \textbf{A3} hold, $n\ge T$, and let $Y_{t_k}=Y^{\epsilon}_{t_k}$ be the chain \eqref{eq:chainpert} with errors satisfying \eqref{eq:markov}, under the assumptions of Theorem~\ref{thm:chaindif}. Then, there exists a probability space supporting versions of the continuous diffusion process $X_t$ and the discrete Markov chain $Y_{t_k}$ such that the probability of their paths coinciding on the macroscopic grid satisfies:
\begin{equation*}
\begin{aligned}
&\PP\left(X_{\tau_k}=Y_{\tau_k} \ \text{for all } k=1,\dots,m\right)\\
&\qquad \ge 1 - C_4(T)\,m\left[\ln\Big(\frac{en}{m}\Big)\Big(\frac{T}{m}\Big)^{\frac{\gamma_2-1}{2}}\Big(\Big(\frac{T}{n}\Big)^{\min(\frac{\gamma_2}{2},\alpha_2,\beta)}+\Delta_a+\Delta_b\Big)+\Big(\frac{m}{n}\Big)^{\frac12}\right],
\end{aligned}
\end{equation*}
with $\Delta_a$, $\Delta_b$ as in Theorem~\ref{thm:chaindif}.
\end{proposition}

\subsection{Local fluctuations of the processes} \label{sec:results3}

In this subsection we will state bounds on the local fluctuations of the diffusion process, the degenerate diffusion and the Markov process, respectively. These bounds can be used to get the results of Subsection \ref{sec:results1} from the couplings on grids shown in the last subsection.
They bound the difference of the values of each process between points in $[0,T]$ and neighboring grid points. Thus they allow to go from the grid to uniform strong approximations over $[0,T]$.

\begin{proposition}\label{thm:supdiff}
Assume that conditions \textbf{A1}, \textbf{A2}, \textbf{A3} hold, and let $X$, $X^{\epsilon}$ be diffusions \eqref{eq:sde} and \eqref{eq:sdepert}. Then for every $m\in\mathbb{N}$ and every $1/2<\alpha<1$, with  the macro-grid $\tau_i=iT/m$:
\begin{eqnarray*}
&&\PP\left(\max_{0\le k\le m-1}\sup_{t\in[\tau_k,\tau_{k+1}]} |X_t-X_{\tau_k}| > m^{\alpha-1}+\frac{KT}{m}\right) \le 2dm\,\exp\left(-\frac{m^{2\alpha-1}}{2dK^2T}\right),\\
&&\PP\left(\max_{0\le k\le m-1}\sup_{t\in[\tau_k,\tau_{k+1}]} |X_t^{\epsilon}-X^{\epsilon}_{\tau_k}| > m^{\alpha-1}+\frac{KT}{m}\right) \le 2dm\,\exp\left(-\frac{m^{2\alpha-1}}{2dK^2T}\right).
\end{eqnarray*}
\end{proposition}

\begin{proposition}\label{thm:supdegdif}
Assume \textbf{A1}, \textbf{A2}, \textbf{AD3} hold.
Then for degenerate diffusions $\widetilde Z_{t}=(\widetilde{X}_{t},\widetilde{Y}_{t})$ and $\widetilde Z^{\epsilon}_{t}=(\widetilde{X}^{(\epsilon)}_{t},\widetilde{Y}^{(\epsilon)}_{t})$
(defined in  \eqref{eq:degsde} and \eqref{eq:degsdepert}), for every $m\in\mathbb{N}$ and every $1/2<\alpha<1$:
\begin{eqnarray*}
&&\PP\left(\max_{0\le k\le m-1}\sup_{t\in[\tau_k,\tau_{k+1}]} |\widetilde X_t-\widetilde X_{\tau_k}| > m^{\alpha-1}+\frac{KT}{m}\right) \le 2dm\,\exp\left(-\frac{m^{2\alpha-1}}{2dK^2T}\right),\\
&&\PP\left(\max_{0\le k\le m-1}\sup_{t\in[\tau_k,\tau_{k+1}]} |\widetilde X^{\epsilon}_t-\widetilde X_{\tau_k}^{\epsilon}| > m^{\alpha-1}+\frac{KT}{m}\right) \le 2dm\,\exp\left(-\frac{m^{2\alpha-1}}{2dK^2T}\right).
\end{eqnarray*}
\end{proposition}

\begin{proposition}\label{thm:chainfluct}
Let $m$ divide $n$ (so that $\Tm\subset\Tn$), and let $Y_{t_{l+1}}=Y_{t_l}+h\,b_l+\sqrt h\,\Sigma_l\,\eta_{l+1}$, linearly interpolated, where $b_l$ and $\Sigma_l$ are $\F_l$-measurable with $|b_l|\le K$, $|\Sigma_l|\le\kappa$, and, conditionally on $\F_l$, $\eta_{l+1}$ is centered with a density bounded by $C_*Q_r$ for some $r>d+1$. Then for every $R>0$ and every $\frac12<\alpha<1$
\begin{equation*}
\PP\Big\{\max_{0\le k\le m-1}\sup_{t\in[\tau_k,\tau_{k+1}]} \big|Y_{t}-Y_{\tau_k}\big| \ge m^{\alpha-1}+\frac{KT}{m}+\frac{c\sqrt{Tn}}{m\,R^{r-d-1}}\Big\} \le 2dm\,e^{-\frac{m^{2\alpha-1}}{2dR^2T}} + \frac{c\,n}{R^{r-d}},
\end{equation*}
where $c:=C_*\,c_r\,\omega_d\,\kappa^{r-d}/(r-d-1)$. This applies to the chains \eqref{eq:chain}, \eqref{eq:chainpert} with independent errors satisfying \textbf{A5}, with $\Sigma_l=\sigma(t_l,Y_{t_l})$ (resp.\ $\sigma_{\epsilon}$), $\eta=\xi^n$, $\kappa=K$, $r=M$, $C_*=C_{\xi}$, and to the chain of Theorem~\ref{thm:chaindif}, with $\Sigma_l=I$, $\eta_{l+1}=\sigma_{\epsilon}(t_l,Y^{\epsilon}_{t_l})\xi^n_{l+1}$, $\kappa=1$, $r=S$, $C_*=C_q$.
\end{proposition}

\begin{proposition}\label{thm:sup}
Assume the assumptions of Theorem~\ref{thm:chaindif}, let $m$ divide $n$, and let $X_t$ and the (linearly interpolated) Markov chain $Y_t=Y^{\varepsilon}_t$ of \eqref{eq:chainpert} with errors satisfying \eqref{eq:markov} be defined on a common probability space (for instance, the one of Proposition~\ref{thm:grid}). Then for every $R>0$ and every $1/2<\alpha<1$, with $c=C_qc_S\omega_d/(S-d-1)$:
\begin{eqnarray*}
&&\PP\left(\max_{0\le k\le m-1}\sup_{t \in[\tau_k,\tau_{k+1}]} |(X_t-Y_t)- (X_{\tau_k}-Y_{\tau_k})| \ge 2m^{\alpha-1}+\frac{2KT}{m}+\frac{c\sqrt{Tn}}{m\,R^{S-d-1}}\right)\\
&& \qquad  \le 2dm\exp\left(-\frac{m^{2\alpha-1}}{2dK^2T}\right)+2dm\exp\left(-\frac{m^{2\alpha-1}}{2dR^2T}\right)+\frac{c\,n}{R^{S-d}}.
\end{eqnarray*}
\end{proposition}

\subsection{Proof of the results in Subsection \ref{sec:results1}} \label{sec:results4} In this subsection we will give proofs for the results in Subsection \ref{sec:results1}. The results of Subsections \ref{sec:results2} and  \ref{sec:results3} are shown in Section  \ref{sec:proofs}.

\emph{From grid couplings to path couplings.} Each of Propositions~\ref{thm:pert}, \ref{thm:deg}, \ref{thm:grid} provides a coupling of the vectors of values of two Markov processes on the macro-grid. We complete it to a coupling of the whole paths by sampling, given the coupled grid values, the path of each process between consecutive macro-points from its conditional (bridge) distribution given the values at these points, conditionally independently for the two processes. By the Markov property each completed process has the correct law, so that the fluctuation bounds of Subsection~\ref{sec:results3} apply to it, and the coincidence probabilities on the grid are unchanged.

\begin{proof}[Proof of Theorem~\ref{thm:dif}]
Let $A=\{X_{\tau_k}=X^{\epsilon}_{\tau_k},\ k=1,\dots,m\}$; since both processes start at the same point, also $X_{\tau_0}=X^{\epsilon}_{\tau_0}$. For $s\in[\tau_k,\tau_{k+1}]$, on $A$ we have $|X_s-X^{\epsilon}_s|\le|X_s-X_{\tau_k}|+|X^{\epsilon}_s-X^{\epsilon}_{\tau_k}|$, so the event $\{|X_s-X^{\epsilon}_s|>2m^{\alpha-1}+2KTm^{-1}\}\cap A$ is contained in the union of the events $\{\sup_{t\in[\tau_k,\tau_{k+1}]}|X_t-X_{\tau_k}|>m^{\alpha-1}+KTm^{-1}\}$ and $\{\sup_{t\in[\tau_k,\tau_{k+1}]}|X^{\epsilon}_t-X^{\epsilon}_{\tau_k}|>m^{\alpha-1}+KTm^{-1}\}$. Hence, by Propositions~\ref{thm:pert} and~\ref{thm:supdiff},
\begin{equation*}
\begin{aligned}
\PP\left(\sup_{s\in[0,T]}|X_s-X^{\epsilon}_s|>2m^{\alpha-1}+\frac{2KT}{m}\right)&\le\PP(A^c)+2\cdot2dm\,e^{-\frac{m^{2\alpha-1}}{2dK^2T}}\\
&\le C_1(q,T)\,m\,\Delta_{\epsilon,\gamma,q}+4dm\,e^{-\frac{m^{2\alpha-1}}{2dK^2T}}.
\end{aligned}\qedhere
\end{equation*}
\end{proof}

\begin{proof}[Proof of Theorem~\ref{thm:degdif}]
Identical to the proof of Theorem~\ref{thm:dif}, with Propositions~\ref{thm:deg} and~\ref{thm:supdegdif} in place of Propositions~\ref{thm:pert} and~\ref{thm:supdiff}.
\end{proof}

\begin{proof}[Proof of Theorem~\ref{thm:chain}]
Let $A=\{Y_{\tau_k}=Y^{\epsilon}_{\tau_k},\ k=1,\dots,m\}$ for the coupling of Proposition~\ref{thm:pert} (chains). On $A$, for $s\in[\tau_k,\tau_{k+1}]$, $|Y_s-Y^{\epsilon}_s|\le|Y_s-Y_{\tau_k}|+|Y^{\epsilon}_s-Y^{\epsilon}_{\tau_k}|$. Applying Proposition~\ref{thm:chainfluct} to $Y$ and to $Y^{\epsilon}$ (with $\kappa=K$, $r=M$, $C_*=C_{\xi}$, so that $c=C_{\xi}c_M\omega_dK^{M-d}/(M-d-1)$),
\begin{equation*}
\PP\left(\sup_{s\in[0,T]}|Y_s-Y^{\epsilon}_s|\ge2m^{\alpha-1}+\frac{2KT}{m}+\frac{2c\sqrt{Tn}}{m\,R^{M-d-1}}\right)\le\PP(A^c)+2\left[2dm\,e^{-\frac{m^{2\alpha-1}}{2dR^2T}}+\frac{c\,n}{R^{M-d}}\right],
\end{equation*}
and Proposition~\ref{thm:pert} (chains, $q=\infty$) bounds $\PP(A^c)$ by $C_3(T)\,m\,\Delta_{\epsilon,\gamma,\infty}$.
\end{proof}

\begin{proof}[Proof of Theorem~\ref{thm:chaindif}]
Let $B=\{X_{\tau_k}=Y^{\epsilon}_{\tau_k},\ k=1,\dots,m\}$ for the coupling of Proposition~\ref{thm:grid} (completed to paths as above), so that $\PP(B^c)$ is bounded by the last term in Theorem~\ref{thm:chaindif}. On $B$, for $s\in[\tau_k,\tau_{k+1}]$, $X_s-Y^{\epsilon}_s=(X_s-Y^{\epsilon}_s)-(X_{\tau_k}-Y^{\epsilon}_{\tau_k})$. Therefore the probability in Theorem~\ref{thm:chaindif} is at most $\PP(B^c)$ plus the probability in Proposition~\ref{thm:sup}, which gives the claim.
\end{proof}

\section {Proofs} \label{sec:proofs}
We give a detailed proof for the case of a Markov chain and a diffusion (Propositions~\ref{thm:grid} and~\ref{thm:sup}). Propositions~\ref{thm:pert} and~\ref{thm:deg} are proved by the same scheme, with the single-step density stability bound replaced by the corresponding known estimates for the perturbed processes; we indicate the necessary substitutions at the end of the paper. The proofs of Propositions~\ref{thm:chainfluct} and~\ref{thm:supdiff} are given in the two concluding subsections.

\subsection{Stability of transition densities: chain and diffusion}

We consider a Markov chain $Y^n_{t_k}$ defined on the micro-grid $t_k=kh$, $h=T/n$, $k=0,\dots,n$, taking values in $\R^d$; the chain \eqref{eq:chainpert} of Theorem~\ref{thm:chaindif} is of this form with $\xi^{\epsilon,n}_{i+1}=\sigma_{\epsilon}(t_i,Y^{\epsilon}_{t_i})\xi^n_{i+1}$. The dynamics of the chain are as follows:
\begin{equation*}
Y^n_{t_{i+1}} = Y^n_{t_i} + h\,b_{\epsilon}(t_i,Y^n_{t_i}) + \sqrt{h}\,\xi^{\epsilon,n}_{i+1}, \qquad i=0,\dots,n-1,
\end{equation*}
where the drift is a function $b_{\epsilon}:[0,T]\times\R^d\to\R^d$, and the family of error terms $\xi^{\epsilon,n}$ satisfies the following Markov condition:
\begin{equation*}
\mathcal{L}\left(\xi^{\epsilon,n}_{i+1}\mid Y^n_{t_i}=y_i,\dots\right) = q^{\epsilon,n}_{t_i,y_i}(\cdot).
\end{equation*}
The probability distributions corresponding to the transition densities $q^{\epsilon,n}_{t_i,x_i}(\cdot)$ have zero mean for all values of $n,t_i,x_i$, and the corresponding conditional covariance matrices are defined by the relation:
\begin{equation*}
\int_{\R^d} z_k z_l\, q^{\epsilon,n}_{t,x}(z)\,dz = a^{\epsilon}_{kl}(t,x).
\end{equation*}
We assume that the chain coefficients $b_{\epsilon}$ and $a^{\epsilon}$ satisfy the conditions \textbf{A1}, \textbf{A2}, \textbf{A3} with the same constants (with $a^{\epsilon}$ in place of $a=\sigma\sigma^*$), and that the family of densities $q^{\epsilon,n}_{t,x}$ satisfies \textbf{A4} (with $q^{\epsilon,n}$ in place of $q^n$, uniformly in $n$). Let $p(t_i,t_j,x,y)$ and $p_{\epsilon}(t_i,t_j,x,y)$ be the transition densities of the diffusion process \eqref{eq:sde} and of the Markov chain, respectively. We measure the distance between the coefficients of the chain and of the diffusion by
\begin{equation*}
\Delta_a := \sup_{t\in[0,T],\,x\in\R^d}\left|a(t,x)-a^{\epsilon}(t,x)\right|, \qquad
\Delta_b := \sup_{t\in[0,T],\,x\in\R^d}\left|b(t,x)-b_{\epsilon}(t,x)\right|,
\end{equation*}
and set, with $\varrho:=\min(\gamma_2/2,\alpha_2,\beta)$,
\begin{equation}\label{eq:deltan}
\Deln := h^{\varrho}+\Delta_a+\Delta_b, \qquad h=\frac{T}{n};
\end{equation}
for $T=1$ the first term is $n^{-\varrho}$. Denote the scaled kernel:
\begin{equation*}
\widehat{Q}_S\!\left(\frac{z}{\sqrt{t}}\right) := t^{-d/2}\,Q_S\!\left(\frac{z}{\sqrt{t}}\right);
\end{equation*}
in all density estimates below we keep, with a slight abuse of notation, the symbol $Q_S$ for the scaled kernel $\widehat{Q}_S$. We assume $n\ge T$, i.e.\ $h\le1$, and record the elementary inequalities, used repeatedly below without further comment:
\begin{equation}\label{eq:elementary}
h^{1/2}\le h^{\varrho}\le\Deln, \qquad h^{\gamma_2/2}\le\Deln, \qquad h^{-1/2}\,(t_j-t_i)^{1/2}=\sqrt{j-i}\ge1.
\end{equation}
The key ingredient of the coupling construction is the following comparison of the transition densities of the chain and of the diffusion.

\begin{proposition}\label{thm:density}
Assume that conditions \textbf{A1}, \textbf{A2}, \textbf{A3}, \textbf{A4} hold and $n\ge T$. For all pairs $(t_i,x)$, $(t_j,y)$ with $0\le t_i<t_j\le T$,
\begin{equation}\label{eq:densitycomp}
\begin{aligned}
|p(t_i,t_j,x,y)-p_{\epsilon}(t_i,t_j,x,y)| &\le C_4(T)\left[\ln\left(e(j-i)\right)(t_j-t_i)^{\frac{\gamma_2-1}{2}}\Deln+(j-i)^{-1/2}\right] \\
&\qquad\times Q_{S-d-6}\!\left(\frac{x-y}{\sqrt{t_j-t_i}}\right),
\end{aligned}
\end{equation}
where the constant $C_4(T)$ depends only on $T$ and on the parameters in the assumptions; it is given explicitly in \eqref{eq:finalT}.
\end{proposition}

\begin{remark}[On the term $(j-i)^{-1/2}$]\label{rem:llt}
The term $(j-i)^{-1/2}$ in \eqref{eq:densitycomp} comes from the term $r=0$ of the parametrix expansions, namely from the comparison \eqref{eq:lltpart} of the frozen chain density $\ptileps$ with the Gaussian density with the same mean and covariance. By the local limit theorem (Theorem~19.3 of \cite{bhattacharya1976}) this difference is of order $h^{1/2}(t_j-t_i)^{-1/2}=(j-i)^{-1/2}$, the order of the first Edgeworth correction after $j-i$ steps. The term cannot be absorbed into $(t_j-t_i)^{\frac{\gamma_2-1}{2}}\Deln$, and in general it cannot be removed: for $j=i+1$ the density $p_{\epsilon}(t_i,t_{i+1},x,\cdot)$ is the rescaled density of a single error, so for non-Gaussian errors $|p-p_{\epsilon}|(t_i,t_{i+1},x,y)$ is of the order $Q_{S}((x-y)/\sqrt h)$ (with the normalisation of this section) and does not tend to zero as $n\to\infty$. For Gaussian errors the frozen chain density is Gaussian and the term disappears. In the coupling construction the bound is used only for $j-i=n/m$, where the term equals $(m/n)^{1/2}$; this is the origin of the condition $l>3$ in Remark~\ref{rem:correction}.
\end{remark}
We now proceed to the proof of Proposition~\ref{thm:density}.
\subsection{Parametrix}

\paragraph{Time horizon, grid and constants.}
Throughout the proof of Proposition~\ref{thm:density} the horizon $T>0$ is arbitrary and fixed, the micro-grid is $t_k=kh$, $k=0,1,\dots,n$, with step $h=T/n$, and all time pairs satisfy $0\le t_i<t_j\le T$, so that
\begin{equation*}
j-i=\frac{t_j-t_i}{h}\le n .
\end{equation*}
We assume $h\le1$, i.e.\ $n\ge T$; this is used only in \eqref{eq:elementary}.
\begin{itemize}
\item $C_0\ge1$ denotes a constant depending only on $d$, $S$, $\gamma_2$ and the constants in \textbf{A1}--\textbf{A4}, but \emph{not} on $T$, $n$, $h$, $i$, $j$, $x$, $y$; its value may change from line to line.
\item $K=K_1\vee K_2$ is the bound of the coefficients from \textbf{A1}.
\item $\mathcal{E}_{a,b}$ is the Mittag-Leffler function \eqref{eq:MLdef}, with the properties \eqref{eq:ML}.
\end{itemize}

Fix the terminal pair $(t_j,y)$. We use the frozen objects of \cite{konakovmammen2000,konakov2017stability}: the frozen diffusion
\begin{equation}\label{eq:frozendiff}
d\widetilde{X}_u = b(u,y)\,du + \sigma(u,y)\,\de W_u, \qquad t_i\le u\le t_j,
\end{equation}
whose transition density $\ptil(t_i,t_j,x,\cdot)$ is Gaussian with mean and covariance
\begin{equation}\label{eq:frozenmean}
\mtil_{t_i,t_j}(x) = x + \int_{t_i}^{t_j} b(u,y)\,du, \qquad
\widetilde{C}_{t_i,t_j} = \int_{t_i}^{t_j} a(u,y)\,du,
\end{equation}
and the frozen Markov chain
\begin{equation}\label{eq:frozenchain}
\widetilde{X}^{\epsilon}_{t_{k+1}} = \widetilde{X}^{\epsilon}_{t_k} + h\,b_{\epsilon}(t_k,y) + \sqrt{h}\,\widetilde{\xi}^{\epsilon,n}_{k+1}, \qquad
\mathcal{L}\left(\widetilde{\xi}^{\epsilon,n}_{k+1}\mid\cdots\right) = q^{\epsilon,n}_{t_k,y}(\cdot),
\end{equation}
with transition density $\ptileps(t_i,t_j,x,\cdot)$. The infinitesimal operators of \eqref{eq:sde} and \eqref{eq:frozendiff} at time $u\in[0,T]$, for $\phi\in C_0^2(\R^d,\R)$, $z\in\R^d$, are
\begin{align*}
L_u\phi(z,y) &= \frac12\Tr\left(a(u,z)D_z^2\phi(z,y)\right) + \langle b(u,z),D_z\phi(z,y)\rangle, \\
\widetilde{L}_u\phi(z,y) &= \frac12\Tr\left(a(u,y)D_z^2\phi(z,y)\right) + \langle b(u,y),D_z\phi(z,y)\rangle,
\end{align*}
with the analogous discrete operators for the chain and its frozen version at time step $t_k=kh$:
\begin{align*}
L^{\epsilon}_{t_k}\phi(x) &:= h^{-1}\left(\E\left[\phi(Y^n_{t_{k+1}})\mid Y^n_{t_k}=x\right] - \phi(x)\right), \\
\widetilde{L}^{\epsilon}_{t_k}\phi(x) &:= h^{-1}\left(\E\left[\phi(\widetilde{X}^{\epsilon}_{t_{k+1}})\mid \widetilde{X}^{\epsilon}_{t_k}=x\right] - \phi(x)\right).
\end{align*}
The parametrix kernels are
\begin{equation*}
H(t,s,x,y) = \left(L_t-\widetilde{L}_t\right)\ptil(t,s,x,y), \qquad
H^{\epsilon}(t_i,t_j,x,y) = \left(L^{\epsilon}_{t_i}-\widetilde{L}^{\epsilon}_{t_i}\right)\ptileps(t_{i+1},t_j,\cdot,y)(x),
\end{equation*}
where for $j=i+1$ we interpret $\ptileps(t_j,t_j,\cdot,y)=\delta_y$, i.e.\ $H^{\epsilon}(t_{j-1},t_j,x,y)=h^{-1}\big(p_{\epsilon}(t_{j-1},t_j,x,y)-\ptileps(t_{j-1},t_j,x,y)\big)$. (The shift $t_i\mapsto t_{i+1}$ in the frozen density is the convention of \cite{konakovmammen2000} and \cite[Section~2.2]{konakov2017stability}; it is suppressed in the notation of \cite[Section~2.1]{bitter2025}, but without it the discrete expansion below fails already for $j=i+1$.)
with the continuous space--time convolution
\begin{equation*}
(f\otimes g)(t,s,x,y) = \int_t^s du \int_{\R^d} dz\, f(t,u,x,z)\,g(u,s,z,y)
\end{equation*}
and its discrete counterpart
\begin{equation*}
(f\otimes_{\epsilon} g)(t_i,t_j,x,y) = \sum_{k=0}^{j-i-1}h\int_{\R^d} f(t_i,t_{i+k},x,z)\,g(t_{i+k},t_j,z,y)\,dz.
\end{equation*}
The transition densities admit the parametrix expansions (see \cite{konakovmammen2000} and \cite[Proposition~2.1]{konakov2017stability} for the continuous case; for unbounded drifts, \cite[Section~2.2.1]{menozzipescezhang2021} uses instead a proxy frozen along the backward flow and a truncated series with remainder $p\otimes H^{\otimes N}$, and \cite{konakovmammen2000}, \cite[Proposition~2.2]{konakov2017stability} for the discrete one, with the kernel $H^{\epsilon}$ defined as above)
\begin{equation}\label{eq:parametrix}
p = \sum_{r=0}^{\infty}\ptil\otimes H^r, \qquad
p_{\epsilon} = \sum_{r=0}^{j-i}\ptileps\otimes_{\epsilon} H^{\epsilon,r},
\end{equation}
where $H^1=H$, $H^r=H^{r-1}\otimes H$, and analogously in the discrete case. Following Section~3 of \cite{bitter2025} we also introduce the auxiliary kernel
\begin{equation*}
\Htil(t_i,t_j,x,y) := \left(L_{t_i}-\widetilde{L}_{t_i}\right)\ptileps(t_i,t_j,x,y),
\end{equation*}
built from the continuous operators applied to the discrete frozen density. Note that $H$, $H^{\epsilon}$ and $\Htil$ all freeze the coefficients at the same point $y$; this will simplify the comparison of $H^{\epsilon}$ with $\Htil$ in Lemma~\ref{lem:ia}.

The role of the properties of the transport flows (Proposition~1 of \cite{bitter2025}) is taken over by the following elementary observation.

\begin{lemma}\label{lem:mean}
Under \textbf{A1}, for $0\le t_i<t_j\le T$, for the continuous frozen mean \eqref{eq:frozenmean} and for the mean $\mtil^{\epsilon}_{t_i,t_j}(x)=x+\sum_{k=i}^{j-1}h\,b_{\epsilon}(t_k,y)$ of the frozen chain,
\begin{equation*}
\left|\mtil_{t_i,t_j}(x)-x\right| \le K_1(t_j-t_i)\le K(t_j-t_i), \qquad \left|\mtil^{\epsilon}_{t_i,t_j}(x)-x\right| \le K_2(t_j-t_i)\le K(t_j-t_i).
\end{equation*}
Consequently, for every $S'>d$, with $C=(1+K\sqrt{T})^{S'}$, for any $v\in\R^d$ with $|v|\le K(t_j-t_i)$,
\begin{equation}\label{eq:shiftkernel}
C^{-1}Q_{S'}\!\left(\frac{x-y}{\sqrt{t_j-t_i}}\right) \le Q_{S'}\!\left(\frac{x-y+v}{\sqrt{t_j-t_i}}\right) \le C\,Q_{S'}\!\left(\frac{x-y}{\sqrt{t_j-t_i}}\right).
\end{equation}
In particular, for all kernel indices $S'\le S$ used below the constant in \eqref{eq:shiftkernel} is at most $(1+K\sqrt T)^{S}$.
\end{lemma}

\begin{proof}
The first claim is immediate from $|b|\le K_1$, $|b_{\epsilon}|\le K_2$ and $\sum_{k=i}^{j-1}h=t_j-t_i$. For \eqref{eq:shiftkernel}: $|x-y|\le|x-y+v|+K(t_j-t_i)$, hence, since $\sqrt{t_j-t_i}\le\sqrt{T}$,
\begin{equation*}
1+\frac{|x-y|}{\sqrt{t_j-t_i}} \le 1+\frac{|x-y+v|}{\sqrt{t_j-t_i}}+K\sqrt{t_j-t_i} \le (1+K\sqrt{T})\left(1+\frac{|x-y+v|}{\sqrt{t_j-t_i}}\right),
\end{equation*}
which gives the upper bound with $C=(1+K\sqrt{T})^{S'}$; the lower bound is identical.
\end{proof}

We will also use the inequality
\begin{equation}\label{eq:smoothing}
\left(\frac{|z|}{\sqrt{t}}\right)^{\varepsilon} Q_S\!\left(\frac{z}{\sqrt{t}}\right) \le C_0\,Q_{S-\varepsilon}\!\left(\frac{z}{\sqrt{t}}\right), \qquad \varepsilon\in(0,S-d),
\end{equation}
and the Chapman--Kolmogorov-type inequality for polynomial kernels (Proposition~2 of \cite{bitter2025}, see also \cite{konakovmammen2000,konakov2017stability}): for $S>d$ there is $C_0\ge1$ with
\begin{equation}\label{eq:chapman}
\int_{\R^d} Q_S\!\left(\frac{z-x}{\sqrt{t_k-t_i}}\right) Q_S\!\left(\frac{y-z}{\sqrt{t_j-t_k}}\right) dz \le C_0\,Q_S\!\left(\frac{y-x}{\sqrt{t_j-t_i}}\right), \qquad 0\le t_i<t_k<t_j\le T.
\end{equation}
Since the scaled kernels $Q_S(\cdot/\sqrt{t})$ contain the normalising factor $t^{-d/2}$, both \eqref{eq:smoothing} and \eqref{eq:chapman} are invariant under the parabolic scaling $(z,t)\mapsto(\lambda z,\lambda^2t)$; hence their constants do not depend on $T$. For completeness we give the short proofs.

\emph{Proof of \eqref{eq:smoothing}.} With $w=z/\sqrt t$, $|w|^{\varepsilon}(1+|w|)^{-S}\le(1+|w|)^{-(S-\varepsilon)}$, so \eqref{eq:smoothing} holds with $C_0=c_S/c_{S-\varepsilon}$.

\emph{Proof of \eqref{eq:chapman}.} Put $s_1=t_k-t_i$, $s_2=t_j-t_k$, $s=s_1+s_2$, $R=|y-x|$; by symmetry we may assume $s_1\ge s/2$. Write $\mathcal{I}$ for the left-hand side of \eqref{eq:chapman}, and recall that $\int_{\R^d}Q_S(z/\sqrt t)\,dz=1$.
If $R\le\sqrt s$, then $Q_S((z-x)/\sqrt{s_1})\le c_Ss_1^{-d/2}\le c_S2^{d/2}s^{-d/2}\le 2^{d/2+S}Q_S((y-x)/\sqrt s)$, and integrating the second factor gives $\mathcal{I}\le2^{d/2+S}Q_S((y-x)/\sqrt s)$.
If $R>\sqrt s$, split $\R^d$ into $\{|z-x|\ge R/2\}$ and $\{|y-z|\ge R/2\}$. On the first set, $Q_S((z-x)/\sqrt{s_1})\le c_S(s/2)^{-d/2}(1+R/(2\sqrt s))^{-S}\le2^{d/2+S}Q_S((y-x)/\sqrt s)$, since $s/2\le s_1\le s$. On the second set, since $S>d$ and $s_2\le s$,
\begin{equation*}
\begin{aligned}
Q_S\!\left(\frac{y-z}{\sqrt{s_2}}\right)&\le c_Ss_2^{-d/2}\left(\frac{2\sqrt{s_2}}{R}\right)^{S}\le c_S2^S\,s^{\frac{S-d}{2}}R^{-S}\\
&\le c_S4^S\,s^{-\frac d2}\left(1+\frac{R}{\sqrt s}\right)^{-S}=4^SQ_S\!\left(\frac{y-x}{\sqrt s}\right),
\end{aligned}
\end{equation*}
because $R\ge\frac12(\sqrt s+R)$. Integrating the remaining factor in each case gives \eqref{eq:chapman} with $C_0=2\cdot4^S\,2^{d/2}$.

\subsection{Frozen transition densities}

\begin{lemma}[flow-free version of Lemma~1 of \cite{bitter2025}]\label{lem:frozen}
For any multi-index $\nu$ with $|\nu|\le4$ there exists $C_0>0$, independent of $T$, such that for all $0\le t_i<t_j\le T$ and $x,y\in\R^d$
\begin{equation*}
\left|\left(D_x^{\nu}\ptil\right)(t_i,t_j,x,y)\right| + \left|\left(D_x^{\nu}\ptileps\right)(t_i,t_j,x,y)\right| \le \frac{C_0\,(1+K\sqrt T)^{S}}{(t_j-t_i)^{|\nu|/2}}\,Q_{S-d-1-|\nu|}\!\left(\frac{x-y}{\sqrt{t_j-t_i}}\right).
\end{equation*}
\end{lemma}

\begin{proof}
For $\ptil$ the bound follows from the explicit Gaussian form with mean and covariance \eqref{eq:frozenmean}, the covariance being uniformly elliptic, $\lambda^{-1}(t_j-t_i)I\le\widetilde{C}_{t_i,t_j}\le\lambda(t_j-t_i)I$, together with Lemma~\ref{lem:mean} (the Gaussian is centered at distance at most $K(t_j-t_i)$ from $x$, and \eqref{eq:shiftkernel} transfers the kernel to the center $x-y$). For $\ptileps$ one writes
\begin{equation}\label{eq:llt-rep}
\ptileps(t_i,t_j,x,y) = \det V_{i,j}^{-1/2}\, f^{\epsilon}_{i,j}\!\left(V_{i,j}^{-1/2}\left(y-\mtil^{\epsilon}_{t_i,t_j}(x)\right)\right), \qquad
V_{i,j} = \sum_{k=i}^{j-1}h\,a^{\epsilon}(t_k,y),
\end{equation}
where $f^{\epsilon}_{i,j}$ is the density of the normalized sum $S^{\epsilon}_{i,j}=V_{i,j}^{-1/2}\sum_{k=i}^{j-1}\sqrt{h}\,\widetilde{\xi}^{\epsilon,n}_{k+1}$. (Note that in \cite{bitter2025} the normalizing factor $V_{i,j}^{-1/2}$ inside the argument of $f^{\epsilon}_{i,j}$ in \eqref{eq:llt-rep} is misprinted.) The bound for $\ptileps$ follows from the corresponding estimate for the kernel $f^{\epsilon}_{i,j}$ after passing to the new argument $u=V_{i,j}^{-1/2}\big(y-\mtil^{\epsilon}_{t_i,t_j}(x)\big)$: for every multi-index $\nu$ with $|\nu|\le4$,
\begin{equation}\label{eq:kernel-f}
\left|D^{\nu}_u f^{\epsilon}_{i,j}(u)\right| \le C_0\,Q_{S-d-1-|\nu|}(u), \qquad u\in\R^d,
\end{equation}
uniformly in $n$, $\epsilon$, $0\le i<j\le n$ and $y$. Estimate \eqref{eq:kernel-f} is obtained by applying Theorem~19.3 of \cite{bhattacharya1976} exactly as in Lemma~1 of \cite{bitter2025} and Lemma~3.7 of \cite{konakovmammen2000}. To pass to the new argument, note that $\mtil^{\epsilon}_{t_i,t_j}(x)=x+v^{\epsilon}$ with $v^{\epsilon}=\sum_{k=i}^{j-1}h\,b_{\epsilon}(t_k,y)$ independent of $x$, so $D_xu=-V_{i,j}^{-1/2}$, and by \textbf{A2} $\lambda^{-1}(t_j-t_i)I\le V_{i,j}\le\lambda(t_j-t_i)I$. Hence $\det V_{i,j}^{-1/2}\le\lambda^{d/2}(t_j-t_i)^{-d/2}$, $\|V_{i,j}^{-1/2}\|\le\lambda^{1/2}(t_j-t_i)^{-1/2}$ and $|u|\ge\lambda^{-1/2}|y-\mtil^{\epsilon}_{t_i,t_j}(x)|/\sqrt{t_j-t_i}$, and \eqref{eq:kernel-f} gives
\begin{equation*}
\left|\left(D_x^{\nu}\ptileps\right)(t_i,t_j,x,y)\right| \le \frac{C_0}{(t_j-t_i)^{|\nu|/2}}\,Q_{S-d-1-|\nu|}\!\left(\frac{y-\mtil^{\epsilon}_{t_i,t_j}(x)}{\sqrt{t_j-t_i}}\right).
\end{equation*}
Since $|v^{\epsilon}|\le K(t_j-t_i)$ by Lemma~\ref{lem:mean}, \eqref{eq:shiftkernel} transfers the kernel to the center $x-y$.

The constants in the Gaussian bound and in \eqref{eq:kernel-f} do not depend on the length of the time interval (the former by parabolic scaling, the latter being uniform in the number $j-i$ of summands); hence the only dependence on $T$ enters through \eqref{eq:shiftkernel}, i.e.\ through the factor $(1+K\sqrt T)^{S}$.
\end{proof}

The next lemma is the corrected, flow-free version of Lemma~3 of \cite{bitter2025}. We state it in split form, since the two parts propagate differently through the parametrix convolutions.

\begin{lemma}[corrected version of Lemma~3 of \cite{bitter2025}]\label{lem:llt}
Let $q=q(t_i,t_j,x,y)$ be the Gaussian density whose mean and covariance coincide with those of $\ptileps(t_i,t_j,x,\cdot)$. For any $|\nu|\le4$ there exists $C_0>0$, independent of $T$, such that for all $0\le t_i<t_j\le T$, $x,y\in\R^d$:
\begin{align}
\left|\left(D_x^{\nu}\ptil - D_x^{\nu}q\right)(t_i,t_j,x,y)\right| &\le \frac{C_0(1+\sqrt{T})\,(1+K\sqrt T)^{S}\,\Deln}{(t_j-t_i)^{|\nu|/2}}\,Q_{S-d-2-|\nu|}\!\left(\frac{x-y}{\sqrt{t_j-t_i}}\right), \label{eq:gausspart}\\
\left|\left(D_x^{\nu}q - D_x^{\nu}\ptileps\right)(t_i,t_j,x,y)\right| &\le \frac{C_0\,(1+K\sqrt T)^{S}\,h^{1/2}}{(t_j-t_i)^{(|\nu|+1)/2}}\,Q_{S-d-1-|\nu|}\!\left(\frac{x-y}{\sqrt{t_j-t_i}}\right). \label{eq:lltpart}
\end{align}
Consequently, using $h^{1/2}\le\Deln$, $(t_j-t_i)^{1/2}\le\sqrt{T}$ and $(1+\sqrt T)\sqrt T+1\le(1+\sqrt T)^2$,
\begin{equation}\label{eq:combined}
\left|\left(D_x^{\nu}\ptil - D_x^{\nu}\ptileps\right)(t_i,t_j,x,y)\right| \le \frac{C_0(1+\sqrt{T})^2\,(1+K\sqrt T)^{S}\,\Deln}{(t_j-t_i)^{(|\nu|+1)/2}}\,Q_{S-d-2-|\nu|}\!\left(\frac{x-y}{\sqrt{t_j-t_i}}\right).
\end{equation}
\end{lemma}

\begin{proof}
\emph{Proof of \eqref{eq:lltpart}.} By \eqref{eq:llt-rep}, $\ptileps$ is, up to the affine change of variables, the density of the normalized sum $S^{\epsilon}_{i,j}$ of $j-i$ independent centered innovations, and $q$ is the corresponding Gaussian. The local limit theorem (Theorem~19.3 of \cite{bhattacharya1976}, applied to derivatives and uniformly over the triangular arrays at hand as in Lemma~1 of \cite{bitter2025} and Lemma~3.7 of \cite{konakovmammen2000}) gives, for the density $f_N$ of a normalized sum of $N$ terms and $|\nu|\le4$,
\begin{equation*}
\left|D^{\nu}f_N(u)-D^{\nu}\varphi(u)\right| \le C_0\,N^{-1/2}\,Q_{S-d-1-|\nu|}(u),
\end{equation*}
where $\varphi$ is the standard Gaussian density. Here the number of summands is
\begin{equation*}
N = j-i = \frac{t_j-t_i}{h}\le n,
\end{equation*}
so the error factor is $(j-i)^{-1/2}=h^{1/2}(t_j-t_i)^{-1/2}$ (for $N<N_0$ the bound follows from \eqref{eq:kernel-f} and the Gaussian bound, since $N^{-1/2}\ge N_0^{-1/2}$); undoing the normalization produces the scaling factor $(t_j-t_i)^{-d/2}$ contained in $Q$ and the factor $(t_j-t_i)^{-|\nu|/2}$ for the derivatives, and \eqref{eq:shiftkernel} transfers the centering to $x-y$ at the cost of the factor $(1+K\sqrt T)^{S}$. No power of $t_j-t_i$ has been bounded by $1$ here.

\emph{Proof of \eqref{eq:gausspart}.} Both $\ptil$ and $q$ are Gaussian. For the means, by \textbf{A3} and the definition of $\Delta_b$,
\begin{equation*}
\begin{aligned}
\left|\mtil_{t_i,t_j}(x)-\mtil^{\epsilon}_{t_i,t_j}(x)\right|
&= \Bigg|\sum_{k=i}^{j-1}\int_{t_k}^{t_{k+1}}\left(b(u,y)-b_{\epsilon}(t_k,y)\right)du\Bigg| \\
&\le (t_j-t_i)\Delta_b + B(t_j-t_i)h^{\beta} \le C_0(t_j-t_i)\Deln,
\end{aligned}
\end{equation*}
using $|b(u,y)-b(t_k,y)|\le B(u-t_k)^{\beta}\le Bh^{\beta}$ on each cell. For the covariances, likewise,
\begin{equation*}
\Bigg|\int_{t_i}^{t_j}a(u,y)\,du - h\sum_{k=i}^{j-1}a^{\epsilon}(t_k,y)\Bigg| \le (t_j-t_i)\Delta_a + C_0(t_j-t_i)h^{\alpha_2} \le C_0(t_j-t_i)\Deln.
\end{equation*}
(In \cite{bitter2025} the analogous computation involves the flows and yields the extra weight $(1+|y|)$; here no weight appears.)

Denote by $g_{m,\Sigma}$ the Gaussian density with mean $m$ and covariance $\Sigma$. Then $\ptil(t_i,t_j,x,\cdot)=g_{m_0,\Sigma_0}$ and $q(t_i,t_j,x,\cdot)=g_{m_1,\Sigma_1}$ with $m_0=\mtil_{t_i,t_j}(x)$, $m_1=\mtil^{\epsilon}_{t_i,t_j}(x)$, $\Sigma_0=\widetilde C_{t_i,t_j}$, $\Sigma_1=V_{i,j}$. By \textbf{A2}, $\lambda^{-1}\tau I\le\Sigma_0,\Sigma_1\le\lambda\tau I$ with $\tau=t_j-t_i$, and the above estimates read $|m_1-m_0|+|\Sigma_1-\Sigma_0|\le C_0\tau\Deln$. For $\theta\in[0,1]$ put $m_\theta=m_0+\theta(m_1-m_0)$, $\Sigma_\theta=\Sigma_0+\theta(\Sigma_1-\Sigma_0)$; then $\lambda^{-1}\tau I\le\Sigma_\theta\le\lambda\tau I$. Using the identities $\partial_{m}g_{m,\Sigma}(z)=-\nabla_zg_{m,\Sigma}(z)$ and $\frac{d}{d\theta}g_{m,\Sigma_\theta}=\frac12\Tr\big((\Sigma_1-\Sigma_0)D^2_zg_{m,\Sigma_\theta}\big)$ 

\begin{equation*}
g_{m_1,\Sigma_1}(z)-g_{m_0,\Sigma_0}(z)=\int_0^1\left[-\langle m_1-m_0,\nabla_z\rangle+\frac12\Tr\big((\Sigma_1-\Sigma_0)D^2_z\big)\right]g_{m_\theta,\Sigma_\theta}(z)\,d\theta .
\end{equation*}
Both densities depend on $x$ only through $z-x$, so $D^{\nu}_x=(-1)^{|\nu|}D^{\nu}_z$. The Gaussian derivative bound of Lemma~\ref{lem:frozen} (with $|\nu|+1$ and $|\nu|+2\le6$ derivatives; for Gaussian densities the computation in the proof of Lemma~\ref{lem:frozen} holds for any order of differentiation and any kernel index) gives $|D^{\mu}_zg_{m_\theta,\Sigma_\theta}(z)|\le C_0\tau^{-|\mu|/2}Q_{S-d-2-|\nu|}((z-m_\theta)/\sqrt\tau)$, and $|m_\theta-x|\le K\tau$ by Lemma~\ref{lem:mean}, so that \eqref{eq:shiftkernel} transfers the kernel to the center $y-x$ at the cost of $(1+K\sqrt T)^{S}$. Hence
\begin{equation*}
\begin{aligned}
\left|\left(D_x^{\nu}\ptil - D_x^{\nu}q\right)(t_i,t_j,x,y)\right|&\le C_0(1+K\sqrt T)^{S}\left[\frac{\tau\Deln}{\tau^{(|\nu|+1)/2}}+\frac{\tau\Deln}{\tau^{(|\nu|+2)/2}}\right]Q_{S-d-2-|\nu|}\!\left(\frac{x-y}{\sqrt\tau}\right)\\
&\le\frac{C_0(1+\sqrt\tau)(1+K\sqrt T)^{S}\Deln}{\tau^{|\nu|/2}}\,Q_{S-d-2-|\nu|}\!\left(\frac{x-y}{\sqrt\tau}\right),
\end{aligned}
\end{equation*}

\end{proof}

Under \textbf{A1}, \textbf{A2}, \textbf{A3}, Lemma~\ref{lem:frozen} together with \eqref{eq:smoothing} gives the kernel estimates (the analogue of Lemma~2 of \cite{bitter2025}): for $0\le t<s\le T$,
\begin{equation}\label{eq:kernelest}
\begin{aligned}
|H|(t,s,x,y) + |\Htil|(t,s,x,y) &\le \frac{C_{H,T}}{(s-t)^{1-\gamma_2/2}}\,Q_{S-d-4}\!\left(\frac{x-y}{\sqrt{s-t}}\right),\\
|H^{\epsilon}|(t_i,t_j,x,y)&\le \frac{C_{H,T}}{(t_j-t_i)^{1-\gamma_2/2}}\,Q_{S-d-6}\!\left(\frac{x-y}{\sqrt{t_j-t_i}}\right),
\end{aligned}
\end{equation}
where $C_{H,T}:=C_0\,(1+K\sqrt T)^{2S}\,(1+T)^{1-\frac{\gamma_2}{2}}$. 

The bound for $H$ and $\Htil$ uses $|a(t,z)-a(t,y)|\le C_0|z-y|^{\gamma_2}$ with \eqref{eq:smoothing} for the second-order part and $|b(t,z)-b(t,y)|\le2K$ for the first-order part, whose order $(s-t)^{-1/2}$ is bounded by $T^{\frac{1-\gamma_2}{2}}(s-t)^{\gamma_2/2-1}$. For the discrete kernel $H^{\epsilon}$ we expand $(L^{\epsilon}_{t_i}-\widetilde L^{\epsilon}_{t_i})\varphi$, $\varphi=\ptileps(t_{i+1},t_j,\cdot,y)$, to third order as in the proof of Lemma~\ref{lem:ia} below, but now with the coefficients themselves: the first- and second-order terms $\langle b_{\epsilon}(t_i,x)-b_{\epsilon}(t_i,y),D\varphi\rangle$ and $\frac12\Tr((a^{\epsilon}(t_i,x)-a^{\epsilon}(t_i,y))D^2\varphi)$ are estimated as for $H$, and the terms $S_3$, $S_4$ are bounded by $C_0(1+K\sqrt T)^{2S}h\,(t_j-t_i)^{-1}Q_{S-d-6}$, where $h/\tau\le(h/\tau)^{\gamma_2/2}\le T^{1-\gamma_2}\tau^{\gamma_2/2-1}$. For $j=i+1$ the bound $|H^{\epsilon}(t_i,t_{i+1},\cdot)|\le C_0(1+K\sqrt T)^{S}Q_{S-1}$ of the last step in Lemma~\ref{lem:ia} is at most $C_0(1+K\sqrt T)^{S}h^{\gamma_2/2-1}Q_{S-1}$. Since $\max\{T^{\frac{1-\gamma_2}{2}},T^{1-\gamma_2}\}\le(1+T)^{1-\frac{\gamma_2}{2}}$, all these bounds are covered by $C_{H,T}$.

Set
\begin{equation}\label{eq:M_T}
M_T := C_0\,C_{H,T}\,\Gamma(\gamma_2/2). 
\end{equation}
By \eqref{eq:kernelest}, \eqref{eq:chapman} and the Beta integral $\int_t^s(u-t)^{a}(s-u)^{\gamma_2/2-1}du=B(1+a,\gamma_2/2)(s-t)^{a+\gamma_2/2}$, we have the following rule: if $a>-1$ and
\begin{equation*}
|f(t,u,x,z)|\le c\,\frac{(u-t)^{a}}{\Gamma(1+a)}\,Q_{S'}\!\left(\frac{z-x}{\sqrt{u-t}}\right),
\end{equation*}
then for every $m\ge0$ each of $|f\otimes H^m|$, $|f\otimes_{\epsilon}H^m|$, $|f\otimes_{\epsilon}\Htil^m|$, $|f\otimes_{\epsilon}H^{\epsilon,m}|$ is bounded by
\begin{equation}\label{eq:betaconv}
c\,\frac{M_T^{m}\,(s-t)^{a+\frac{m\gamma_2}{2}}}{\Gamma\!\left(1+a+\frac{m\gamma_2}{2}\right)}\,Q_{S''}\!\left(\frac{y-x}{\sqrt{s-t}}\right), \qquad S''=\min(S',S-d-4)
\end{equation}
by induction, $S''=\min(S',S-d-6)$ for $H^{\epsilon}$.

Here and below, in a discrete convolution whose left factor is a kernel we use the convention $H(t_k,t_k,\cdot,\cdot)=\Htil(t_k,t_k,\cdot,\cdot)=H^{\epsilon}(t_k,t_k,\cdot,\cdot)=0$, while the densities satisfy $\ptil(t_k,t_k,x,\cdot)=\ptileps(t_k,t_k,x,\cdot)=p(t_k,t_k,x,\cdot)=\delta_x$.
Applying \eqref{eq:betaconv} with $f=\ptil$ or $f=\ptileps$, $a=0$, $c=C_0(1+K\sqrt T)^{S}$ (Lemma~\ref{lem:frozen}), we obtain by induction
\begin{equation}\label{eq:iterated}
\begin{aligned}
\left|\ptil\otimes H^r\right| + \left|\ptileps\otimes_{\epsilon}\Htil^r\right| &\le C_0(1+K\sqrt T)^{S}\,\frac{\left(C_0C_{H,T}\right)^{r}\,\Gamma^r(\gamma_2/2)}{\Gamma\!\left(1+\frac{r\gamma_2}{2}\right)}\,(s-t)^{\frac{r\gamma_2}{2}}\,Q_{S-d-4}\!\left(\frac{x-y}{\sqrt{s-t}}\right) \\
&= C_0(1+K\sqrt T)^{S}\,\frac{M_T^r\,(s-t)^{\frac{r\gamma_2}{2}}}{\Gamma\!\left(1+\frac{r\gamma_2}{2}\right)}\,Q_{S-d-4}\!\left(\frac{x-y}{\sqrt{s-t}}\right),
\end{aligned}
\end{equation}
and in particular, since $\mathcal{E}_{\gamma_2/2,1}$ is increasing,
\begin{equation}\label{eq:seriesbound}
\sum_{r\ge0}\left|\ptil\otimes H^r\right|(t,s,x,y) \le C_0(1+K\sqrt T)^{S}\,\mathcal{E}_{\gamma_2/2,1}\!\left(M_T\,T^{\gamma_2/2}\right)Q_{S-d-4}\!\left(\frac{x-y}{\sqrt{s-t}}\right).
\end{equation}
The proof of Proposition~\ref{thm:density} is based on the decomposition (3.1) of \cite{bitter2025}:
\begin{equation}\label{eq:decomposition}
\begin{aligned}
|p-p_{\epsilon}|(t_i,t_j,x,y) &\le \Bigg|\sum_{r=0}^{j-i}\ptileps\otimes_{\epsilon}\left(H^{\epsilon,r}-\Htil^r\right)\Bigg| + \Bigg|\sum_{r=0}^{j-i}\ptileps\otimes_{\epsilon}\left(\Htil^r-H^r\right)\Bigg| \\
&\quad + \Bigg|\sum_{r=0}^{j-i}\left(\ptileps-\ptil\right)\otimes_{\epsilon}H^r\Bigg| + \Bigg|\sum_{r=0}^{\infty}\ptil\otimes_{\epsilon}H^r - \ptil\otimes H^r\Bigg| \\
&=: i_a + i_b + ii + iii.
\end{aligned}
\end{equation}
(We have split the term $i$ of \cite{bitter2025} into $i_a+i_b$ for clarity; $i_b$ corresponds to Lemma~3.6 and $i_a$ to Lemma~7 of \cite{bitter2025}.) In $i_a$, $i_b$, $ii$, $iii$ the powers $H^r$, $\Htil^r$ standing to the right of $\otimes_{\epsilon}$ denote the iterated \emph{discrete} convolutions $H\otimes_{\epsilon}\dots\otimes_{\epsilon}H$; they vanish at $(t_i,t_j)$ for $r>j-i$, so that $p^{\otimes_\epsilon}:=\sum_{r\ge0}\ptil\otimes_{\epsilon}H^r$ is a finite sum, and the decomposition is the telescoping of $p-p_{\epsilon}$ through $p^{\otimes_\epsilon}$ and $\sum_r\ptileps\otimes_\epsilon H^r$, $\sum_r\ptileps\otimes_\epsilon\Htil^r$.

\subsection{The term $ii$}

\begin{lemma}\label{lem:ii}
For all $0\le t_i<t_j\le T$ and $x,y\in\R^d$,
\begin{equation}\label{eq:iibound}
\sum_{r=1}^{j-i}\left|\left(\ptileps-\ptil\right)\otimes_{\epsilon}H^r\right|(t_i,t_j,x,y) \le C^{(ii)}_T\,(t_j-t_i)^{\frac{\gamma_2-1}{2}}\Deln\,Q_{S-d-4}\!\left(\frac{x-y}{\sqrt{t_j-t_i}}\right),
\end{equation}
with
\begin{equation*}
C^{(ii)}_T := \sqrt{\pi}\,C_0(1+\sqrt{T})^2\,(1+K\sqrt T)^{S}\,M_T\,\mathcal{E}_{\frac{\gamma_2}{2},\frac{1+\gamma_2}{2}}\!\left(M_T\,T^{\gamma_2/2}\right),
\end{equation*}
and the term $r=0$ satisfies
\begin{equation}\label{eq:r0}
\begin{aligned}
\left|\ptileps-\ptil\right|(t_i,t_j,x,y) &\le C_0(1+K\sqrt T)^{S}\left((1+\sqrt{T})\Deln+(j-i)^{-1/2}\right)Q_{S-d-2}\!\left(\frac{x-y}{\sqrt{t_j-t_i}}\right) \\
&\le C_0(1+\sqrt{T})^2\,(1+K\sqrt T)^{S}\,\Deln\,(t_j-t_i)^{-1/2}\,Q_{S-d-2}\!\left(\frac{x-y}{\sqrt{t_j-t_i}}\right).
\end{aligned}
\end{equation}
\end{lemma}

\begin{proof}
The bound \eqref{eq:r0} is Lemma~\ref{lem:llt} with $|\nu|=0$ (in both the split and the combined form); in the second line we used $\Deln\le\sqrt{T}\,\Deln\,(t_j-t_i)^{-1/2}$, $(j-i)^{-1/2}=h^{1/2}(t_j-t_i)^{-1/2}\le\Deln\,(t_j-t_i)^{-1/2}$ and $(1+\sqrt T)\sqrt T+1\le(1+\sqrt T)^2$.

Let $r=1$. Since $\ptileps(t_i,t_i,x,\cdot)=\ptil(t_i,t_i,x,\cdot)=\delta_x$, the term $k=i$ of the discrete convolution vanishes, and by \eqref{eq:combined}, \eqref{eq:kernelest} and \eqref{eq:chapman},
\begin{align*}
\left|\left(\ptileps-\ptil\right)\otimes_{\epsilon}H\right|(t_i,t_j,x,y)
&\le \sum_{k=i+1}^{j-1}h\int_{\R^d}\left|(\ptileps-\ptil)(t_i,t_k,x,z)\right|\left|H(t_k,t_j,z,y)\right|dz \\
&\le C_0(1+\sqrt{T})^2\,(1+K\sqrt T)^{S}\,C_{H,T}\,\Deln\,Q_{S-d-4}\!\left(\frac{x-y}{\sqrt{t_j-t_i}}\right) \\
&\qquad\times\sum_{k=i+1}^{j-1}h\,(t_k-t_i)^{-1/2}(t_j-t_k)^{\gamma_2/2-1}.
\end{align*}
Splitting the sum at the midpoint of $[t_i,t_j]$ and comparing each half with the integral of the corresponding monotone factor (the constant of this comparison does not depend on $h$ or $T$, by scaling), we obtain
\begin{equation*}
\begin{aligned}
\sum_{k=i+1}^{j-1}h\,(t_k-t_i)^{-1/2}(t_j-t_k)^{\gamma_2/2-1} &\le C_0\int_{t_i}^{t_j}(u-t_i)^{-1/2}(t_j-u)^{\gamma_2/2-1}\,du \\
&= C_0\,\frac{\Gamma(1/2)\,\Gamma(\gamma_2/2)}{\Gamma\!\left(\frac12+\frac{\gamma_2}{2}\right)}\,(t_j-t_i)^{\frac{\gamma_2-1}{2}},
\end{aligned}
\end{equation*}
where the last equality is the Beta function $B(1/2,\gamma_2/2)$. Hence
\begin{equation*}
\begin{aligned}
\left|\left(\ptileps-\ptil\right)\otimes_{\epsilon}H\right|(t_i,t_j,x,y) &\le C_0(1+\sqrt{T})^2\,(1+K\sqrt T)^{S}\,\left(C_0C_{H,T}\right) \\
&\qquad\times\frac{\Gamma(1/2)\,\Gamma(\gamma_2/2)}{\Gamma\!\left(\frac12+\frac{\gamma_2}{2}\right)}\,(t_j-t_i)^{\frac{\gamma_2-1}{2}}\Deln\,Q_{S-d-4}\!\left(\frac{x-y}{\sqrt{t_j-t_i}}\right).
\end{aligned}
\end{equation*}
Assume by induction that
\begin{equation*}
\begin{aligned}
\left|\left(\ptileps-\ptil\right)\otimes_{\epsilon}H^r\right|(t_i,t_j,x,y) &\le C_0(1+\sqrt{T})^2\,(1+K\sqrt T)^{S}\,\left(C_0C_{H,T}\right)^{r} \\
&\qquad\times\frac{\Gamma^r(\gamma_2/2)\,\Gamma(1/2)}{\Gamma\!\left(\frac12+\frac{r\gamma_2}{2}\right)}\,(t_j-t_i)^{\frac{r\gamma_2}{2}-\frac12}\Deln\,Q_{S-d-4}\!\left(\frac{x-y}{\sqrt{t_j-t_i}}\right).
\end{aligned}
\end{equation*}
Then, using \eqref{eq:kernelest}, \eqref{eq:chapman} and the same sum--integral comparison,
\begin{align*}
\left|\left(\ptileps-\ptil\right)\otimes_{\epsilon}H^{r+1}\right|(t_i,t_j,x,y)
&\le C_0(1+\sqrt{T})^2\,(1+K\sqrt T)^{S}\,\left(C_0C_{H,T}\right)^{r+1}\frac{\Gamma^r(\gamma_2/2)\,\Gamma(1/2)}{\Gamma\!\left(\frac12+\frac{r\gamma_2}{2}\right)} \\
&\qquad\times\Deln\,Q_{S-d-4}\!\left(\frac{x-y}{\sqrt{t_j-t_i}}\right)\int_{t_i}^{t_j}(u-t_i)^{\frac{r\gamma_2}{2}-\frac12}(t_j-u)^{\frac{\gamma_2}{2}-1}\,du \\
&= C_0(1+\sqrt{T})^2\,(1+K\sqrt T)^{S}\,\left(C_0C_{H,T}\right)^{r+1}\frac{\Gamma^{r+1}(\gamma_2/2)\,\Gamma(1/2)}{\Gamma\!\left(\frac12+\frac{(r+1)\gamma_2}{2}\right)} \\
&\qquad\times(t_j-t_i)^{\frac{(r+1)\gamma_2}{2}-\frac12}\Deln\,Q_{S-d-4}\!\left(\frac{x-y}{\sqrt{t_j-t_i}}\right),
\end{align*}
since
\begin{equation*}
\begin{aligned}
\int_{t_i}^{t_j}(u-t_i)^{\frac{r\gamma_2}{2}-\frac12}(t_j-u)^{\frac{\gamma_2}{2}-1}\,du
&= B\!\left(\frac{r\gamma_2}{2}+\frac12,\frac{\gamma_2}{2}\right)(t_j-t_i)^{\frac{(r+1)\gamma_2}{2}-\frac12} \\
&= \frac{\Gamma\!\left(\frac{r\gamma_2}{2}+\frac12\right)\Gamma\!\left(\frac{\gamma_2}{2}\right)}{\Gamma\!\left(\frac{(r+1)\gamma_2}{2}+\frac12\right)}\,(t_j-t_i)^{\frac{(r+1)\gamma_2}{2}-\frac12}.
\end{aligned}
\end{equation*}
(This is \eqref{eq:betaconv} with $a=-1/2$.) Summing over $r\ge1$ and using \eqref{eq:ML} with $\alpha=\gamma_2/2$, $\beta=1/2$, $z=M_T(t_j-t_i)^{\gamma_2/2}$,

(The constant $M_T$ is defined in \eqref{eq:M_T}.)
\begin{equation*}
\begin{aligned}
\sum_{r\ge1}\frac{\Gamma(1/2)\,M_T^{r}\,(t_j-t_i)^{\frac{r\gamma_2}{2}-\frac12}}{\Gamma\!\left(\frac12+\frac{r\gamma_2}{2}\right)}
&= \sqrt{\pi}\,(t_j-t_i)^{-1/2}\,M_T(t_j-t_i)^{\gamma_2/2}\,\mathcal{E}_{\frac{\gamma_2}{2},\frac{1+\gamma_2}{2}}\!\left(M_T(t_j-t_i)^{\gamma_2/2}\right) \\
&\le \sqrt{\pi}\,M_T\,\mathcal{E}_{\frac{\gamma_2}{2},\frac{1+\gamma_2}{2}}\!\left(M_T\,T^{\gamma_2/2}\right)(t_j-t_i)^{\frac{\gamma_2-1}{2}},
\end{aligned}
\end{equation*}
which gives \eqref{eq:iibound}. Thus the series is bounded by its first term multiplied by
\begin{equation*}
\Gamma\!\left(\tfrac{1+\gamma_2}{2}\right)\mathcal{E}_{\frac{\gamma_2}{2},\frac{1+\gamma_2}{2}}\big(M_TT^{\gamma_2/2}\big);
\end{equation*} 
\end{proof}

\subsection{The term $iii$}

Let
\begin{equation}\label{eq:PT}
P_T:=C_0(1+K\sqrt T)^{S}\,\mathcal{E}_{\gamma_2/2,1}\!\left(M_T\,T^{\gamma_2/2}\right).
\end{equation}
By the continuous parametrix expansion \eqref{eq:parametrix} and \eqref{eq:seriesbound}, the transition density of \eqref{eq:sde} satisfies
\begin{equation}\label{eq:pbound}
p(t,s,x,z)\le P_T\,Q_{S-d-4}\!\left(\frac{z-x}{\sqrt{s-t}}\right),\qquad 0\le t<s\le T .
\end{equation}

\begin{lemma}\label{lem:iii}
For all $0\le t_i<t_j\le T$ and $x,y\in\R^d$,
\begin{equation}\label{eq:iiibound}
iii = \Bigg|\sum_{r=0}^{\infty}\ptil\otimes_{\epsilon}H^r - \ptil\otimes H^r\Bigg|(t_i,t_j,x,y)
\le C^{(iii)}_T\ln\left(e(j-i)\right)\Deln\,Q_{S-d-6}\!\left(\frac{x-y}{\sqrt{t_j-t_i}}\right),
\end{equation}
with $C^{(iii)}_T:=4C_0\,P_T^2\,C_{H,T}\,(1+\sqrt T)\,\mathcal{E}_{\gamma_2/2,1}\!\left(M_T\,T^{\gamma_2/2}\right)$.
\end{lemma}

\begin{proof}
Since the continuous series converges absolutely, $p=\ptil+p\otimes H$; and since the discrete powers are iterated discrete convolutions, $p^{\otimes_\epsilon}=\ptil+p^{\otimes_\epsilon}\otimes_{\epsilon}H$. Hence
\begin{equation*}
p-p^{\otimes_\epsilon}=\mathcal{D}+(p-p^{\otimes_\epsilon})\otimes_{\epsilon}H,\qquad \mathcal{D}:=p\otimes H-p\otimes_{\epsilon}H ,
\end{equation*}
which is the identity (3.7) of \cite{bitter2025}. Since $(p-p^{\otimes_\epsilon})(t_i,t_i,x,\cdot)=\delta_x-\delta_x=0$, iterating it terminates after $j-i$ steps:
\begin{equation}\label{eq:iii-iter}
iii=|p-p^{\otimes_\epsilon}|(t_i,t_j,x,y)=\Bigg|\sum_{r=0}^{j-i-1}\left(\mathcal{D}\otimes_{\epsilon}H^r\right)(t_i,t_j,x,y)\Bigg| .
\end{equation}
It remains to prove
\begin{equation}\label{eq:Dbound}
|\mathcal{D}|(t_i,t_j,x,y)\le 4C_0\,P_T^2\,C_{H,T}(1+\sqrt T)\ln\left(e(j-i)\right)\Deln\,Q_{S-d-6}\!\left(\frac{x-y}{\sqrt{t_j-t_i}}\right);
\end{equation}
then \eqref{eq:iii-iter} and \eqref{eq:betaconv} with $a=0$ (note that $\ln(e(k-i))\le\ln(e(j-i))$ for $k\le j$) give \eqref{eq:iiibound}, the sum over $r$ being bounded by $\mathcal{E}_{\gamma_2/2,1}(M_TT^{\gamma_2/2})$.

 Write $\tau=t_j-t_i$, $H(u):=H(u,t_j,\cdot,y)$, and recall $p(t_i,t_i,x,\cdot)=\delta_x$. Then
\begin{equation*}
\mathcal{D}=\sum_{k=i}^{j-1}\int_{t_k}^{t_{k+1}}du\int_{\R^d}\left[p(t_i,u,x,z)H(u,t_j,z,y)-p(t_i,t_k,x,z)H(t_k,t_j,z,y)\right]dz=:\sum_{k=i}^{j-1}\mathcal{D}_k .
\end{equation*}
\emph{The last cell} $k=j-1$: by \eqref{eq:kernelest}, \eqref{eq:pbound} and \eqref{eq:chapman},
\begin{equation*}
\begin{aligned}
|\mathcal{D}_{j-1}|&\le C_0P_TC_{H,T}\left[\int_{t_{j-1}}^{t_j}(t_j-u)^{\frac{\gamma_2}{2}-1}du+h\cdot h^{\frac{\gamma_2}{2}-1}\right]Q_{S-d-4}\!\left(\frac{x-y}{\sqrt\tau}\right)\\
&\le C_0P_TC_{H,T}\left(1+\frac2{\gamma_2}\right)h^{\frac{\gamma_2}{2}}Q_{S-d-4}\!\left(\frac{x-y}{\sqrt\tau}\right).
\end{aligned}
\end{equation*}
For $k\le j-2$ and $u\in[t_k,t_{k+1}]$ we have $t_j-t_k\le2(t_j-u)$, so all kernels evaluated at the times $t_j-u$ and $t_j-t_k$ are comparable up to constants depending only on $d,S$. We split $\mathcal{D}_k=\mathcal{A}_k+\mathcal{B}_k$,
\begin{equation*}
\begin{aligned}
\mathcal{A}_k&:=\int_{t_k}^{t_{k+1}}\!\!du\int p(t_i,u,x,z)\left[H(u,t_j,z,y)-H(t_k,t_j,z,y)\right]dz,\\
\mathcal{B}_k&:=\int_{t_k}^{t_{k+1}}\!\!du\int\left[p(t_i,u,x,z)-p(t_i,t_k,x,z)\right]H(t_k,t_j,z,y)\,dz .
\end{aligned}
\end{equation*}

By definition,
\begin{equation*}
H(u,t_j,z,y)=\frac12\Tr\left((a(u,z)-a(u,y))D^2_z\ptil(u,t_j,z,y)\right)+\langle b(u,z)-b(u,y),D_z\ptil(u,t_j,z,y)\rangle ,
\end{equation*}
and $u$ enters through the coefficients and through $\ptil(u,t_j,\cdot,y)$. (a) \emph{Coefficients.} By \textbf{A3}, $|a(u,\cdot)-a(t_k,\cdot)|\le C_0h^{\alpha_2}$ and $|b(u,\cdot)-b(t_k,\cdot)|\le Bh^{\beta}$; with Lemma~\ref{lem:frozen} this part of $|H(u)-H(t_k)|$ is at most $C_0(1+K\sqrt T)^{S}[h^{\alpha_2}(t_j-u)^{-1}+h^{\beta}(t_j-u)^{-1/2}]Q_{S-d-3}$. (b) \emph{Frozen density.} $\ptil(s,t_j,\cdot,y)$ is Gaussian with mean shift $\int_s^{t_j}b(r,y)\,dr$ and covariance $\int_s^{t_j}a(r,y)\,dr$; between $s=t_k$ and $s=u$ these parameters change by at most $Kh$ and $\lambda h$. The Gaussian interpolation used in the proof of \eqref{eq:gausspart} gives, for $|\nu|=1,2$,
\begin{equation*}
\begin{aligned}
&\left|D^{\nu}_z\ptil(u,t_j,z,y)-D^{\nu}_z\ptil(t_k,t_j,z,y)\right|\\
&\qquad\le C_0(1+K\sqrt T)^{S}(1+\sqrt T)\,h\,(t_j-u)^{-\frac{|\nu|}{2}-1}Q_{S-d-1-|\nu|}\!\left(\frac{z-y}{\sqrt{t_j-u}}\right).
\end{aligned}
\end{equation*}
Multiplied by $|a(t_k,z)-a(t_k,y)|\le C_0|z-y|^{\gamma_2}$ (converted by \eqref{eq:smoothing}) and by $|b(t_k,z)-b(t_k,y)|\le2K$, this part of $|H(u)-H(t_k)|$ is at most $C_0(1+K\sqrt T)^{S}(1+\sqrt T)h[(t_j-u)^{\frac{\gamma_2}{2}-2}+(t_j-u)^{-3/2}]Q_{S-d-4}$. Integrating against \eqref{eq:pbound} with \eqref{eq:chapman}, and summing with $t_j-t_{k+1}=m'h$, $m'=j-k-1\ge1$:
\begin{equation*}
\begin{aligned}
\sum_{k=i}^{j-2}|\mathcal{A}_k|&\le C_0P_T(1+K\sqrt T)^{S}(1+\sqrt T)\Big[h^{\alpha_2}\sum_{m'=1}^{j-i}\frac1{m'}+\sqrt T\,h^{\beta}\\
&\qquad+h^{\frac{\gamma_2}{2}}\sum_{m'\ge1}(m')^{\frac{\gamma_2}{2}-2}+h^{\frac12}\sum_{m'\ge1}(m')^{-\frac32}\Big]\,Q_{S-d-4}\!\left(\frac{x-y}{\sqrt\tau}\right),
\end{aligned}
\end{equation*}
where we used $\sum_kh\,(t_j-t_{k+1})^{-1/2}\le2\sqrt\tau$. Hence $\sum_k|\mathcal{A}_k|\le C_0P_T(1+K\sqrt T)^{S}(1+\sqrt T)^2\ln(e(j-i))\Deln\,Q_{S-d-4}$, since $h^{\alpha_2},h^{\beta},h^{\gamma_2/2},h^{1/2}\le\Deln$.

By the Chapman--Kolmogorov equation for $p$ (and $p(t_i,t_i,x,\cdot)=\delta_x$ if $k=i$),
\begin{equation*}
\begin{aligned}
&\int\left[p(t_i,u,x,z)-p(t_i,t_k,x,z)\right]H(t_k,t_j,z,y)\,dz=\int p(t_i,t_k,x,w)\,\Phi_u(w)\,dw,\\
&\Phi_u(w):=\int p(t_k,u,w,z)\left[H(t_k,t_j,z,y)-H(t_k,t_j,w,y)\right]dz,
\end{aligned}
\end{equation*}
where we used $\int p(t_k,u,w,z)\,dz=1$. We claim the H\"older bound, with $s:=t_j-t_k$,
\begin{equation}\label{eq:Hholder}
\begin{aligned}
&\left|H(t_k,t_j,z,y)-H(t_k,t_j,w,y)\right|\\
&\qquad\le C_0\,C_{H,T}(1+\sqrt T)\,\frac{|z-w|^{\gamma_2}}{s}\left[Q_{S-d-5}\!\left(\frac{z-y}{\sqrt s}\right)+Q_{S-d-5}\!\left(\frac{w-y}{\sqrt s}\right)\right].
\end{aligned}
\end{equation}
We first show $|H(z)-H(w)|\le C_0(1+K\sqrt T)^{S}|z-w|^{\gamma_2}\big(s^{-1}+s^{-1/2}+s^{-(1+\gamma_2)/2}\big)[Q(z)+Q(w)]$; since $s\le T$, $s^{-1/2}\le\sqrt T\,s^{-1}$ and $s^{-(1+\gamma_2)/2}\le T^{\frac{1-\gamma_2}2}s^{-1}$, this gives \eqref{eq:Hholder}. If $|z-w|\ge\sqrt s$, it follows from the unabsorbed form $|H|\le C_0(1+K\sqrt T)^{S}(s^{\gamma_2/2-1}+s^{-1/2})Q_{S-d-4}$ of \eqref{eq:kernelest}, since $s^{\gamma_2/2-1}\le s^{-1}|z-w|^{\gamma_2}$ and $s^{-1/2}\le s^{-(1+\gamma_2)/2}|z-w|^{\gamma_2}$. If $|z-w|<\sqrt s$, the kernels at the points of the segment $[w,z]$ are comparable, and
\begin{equation*}
\begin{aligned}
|H(z)-H(w)|&\le\tfrac12|a(z)-a(w)|\,|D^2\ptil(z)|+\tfrac12|a(w)-a(y)|\,|D^2\ptil(z)-D^2\ptil(w)| \\
&\quad+|b(z)-b(w)|\,|D\ptil(z)|+|b(w)-b(y)|\,|D\ptil(z)-D\ptil(w)|,
\end{aligned}
\end{equation*}
all at time $t_k$. Here

\begin{equation*}
\begin{aligned}
&|a(z)-a(w)|\le C_0|z-w|^{\gamma_2},\\
&|D^2\ptil(z)-D^2\ptil(w)|\le|z-w|\sup|D^3\ptil|\le C_0(1+K\sqrt T)^{S}|z-w|^{\gamma_2}s^{-1-\gamma_2/2}Q\\
&\qquad\text{(Lemma~\ref{lem:frozen} with $|\nu|=3$ and $|z-w|<\sqrt s$)},\\
&|a(w)-a(y)|\le C_0|w-y|^{\gamma_2}\ \text{is converted into $s^{\gamma_2/2}$ by \eqref{eq:smoothing}},\\
&|b(z)-b(w)|\le\min(2K,B|z-w|)\le\max(2K,B)|z-w|^{\gamma_2}\\
&\qquad\text{(multiplied by $|D\ptil|\le C_0(1+K\sqrt T)^{S}s^{-1/2}Q$)},\\
&|b(w)-b(y)|\,|D\ptil(z)-D\ptil(w)|\le2K\,|z-w|\,C_0(1+K\sqrt T)^{S}s^{-1}Q\\
&\qquad\le C_0(1+K\sqrt T)^{S}|z-w|^{\gamma_2}s^{-(1+\gamma_2)/2}Q;\\
&\text{collecting the powers of $s$ gives the claim.}
\end{aligned}
\end{equation*}

Now, by \eqref{eq:pbound} with $u-t_k\le h$, \eqref{eq:smoothing} and \eqref{eq:chapman}
(note $s\ge2h\ge2(u-t_k)$),
\begin{equation*}
|\Phi_u(w)|\le C_0P_TC_{H,T}(1+\sqrt T)\,\frac{h^{\gamma_2/2}}{s}\,Q_{S-d-6}\!\left(\frac{w-y}{\sqrt s}\right),
\end{equation*}
and by \eqref{eq:pbound} and \eqref{eq:chapman} once more, 

\begin{equation*}
|\mathcal{B}_k|\le C_0P_T^2C_{H,T}(1+\sqrt T)\,h\cdot h^{\gamma_2/2}(t_j-t_k)^{-1}Q_{S-d-6}((x-y)/\sqrt\tau).
\end{equation*}

Since $\sum_{k=i}^{j-2}h(t_j-t_k)^{-1}\le\ln(e(j-i))$ and $h^{\gamma_2/2}\le\Deln$,
\begin{equation*}
\sum_{k=i}^{j-2}|\mathcal{B}_k|\le C_0P_T^2C_{H,T}(1+\sqrt T)\ln\left(e(j-i)\right)\Deln\,Q_{S-d-6}\!\left(\frac{x-y}{\sqrt\tau}\right).
\end{equation*}

Since $(1+K\sqrt T)^{S}\le C_{H,T}$, $1\le P_T$ and $1+\sqrt T\le C_0(1+K\sqrt T)^{S}\le C_0P_T$, each of these groups of terms is bounded by $C_0P_T^2C_{H,T}(1+\sqrt T)\ln(e(j-i))\Deln\,Q_{S-d-6}((x-y)/\sqrt\tau)$, which gives \eqref{eq:Dbound}. This proves the lemma.
\end{proof}

\subsection{The term $i_b$: comparison of $\Htil$ with $H$}

\begin{lemma}\label{lem:ib}
For all $0\le t_i<t_j\le T$ and $x,y\in\R^d$,
\begin{equation}\label{eq:ibbound}
i_b = \Bigg|\sum_{r=0}^{j-i}\ptileps\otimes_{\epsilon}\left(\Htil^r-H^r\right)\Bigg|(t_i,t_j,x,y) \le C^{(i_b)}_T\ln\left(e(j-i)\right)\Deln\,Q_{S-d-5}\!\left(\frac{x-y}{\sqrt{t_j-t_i}}\right),
\end{equation}
with $C^{(i_b)}_T:=B_T\,\mathcal{E}_{\gamma_2/2,1}\!\left(2M_T\,T^{\gamma_2/2}\right)$ and $B_T:=C_0\,(1+K\sqrt T)^{2S}\,(1+\sqrt{T})\left(1+M_T\,T^{\gamma_2/2}\right)$.
\end{lemma}

\begin{proof}
Consider first $r=1$. For intermediate times $t_i\le t_k<t_j$,
\begin{align*}
|\Htil-H|(t_k,t_j,z,y) &\le \sum_{m'=1}^d\left|b_{m'}(t_k,z)-b_{m'}(t_k,y)\right|\left|\frac{\partial(\ptil-\ptileps)}{\partial z_{m'}}\right| \\
&\quad + \frac12\sum_{m',l=1}^d\left|a_{m'l}(t_k,z)-a_{m'l}(t_k,y)\right|\left|\frac{\partial^2(\ptil-\ptileps)}{\partial z_{m'}\partial z_l}\right|
\end{align*}
and we estimate the two parts of $\ptil-\ptileps=(\ptil-q)+(q-\ptileps)$ separately.

\emph{Gaussian part.} By \eqref{eq:gausspart}, \textbf{A3} and \eqref{eq:smoothing}, using $|b(t_k,z)-b(t_k,y)|\le2K_1$ for the drift term (of order $(t_j-t_k)^{-1/2}$) and absorbing it exactly as in \eqref{eq:kernelest},
\begin{equation*}
\left|(L_{t_k}-\widetilde{L}_{t_k})(\ptil-q)\right|(t_k,t_j,z,y) \le \frac{(1+\sqrt{T})\,C_{H,T}\,\Deln}{(t_j-t_k)^{1-\gamma_2/2}}\,Q_{S-d-5}\!\left(\frac{z-y}{\sqrt{t_j-t_k}}\right),
\end{equation*}
so by Lemma~\ref{lem:frozen}, \eqref{eq:chapman},
\begin{equation*}
\begin{aligned}
\left|\ptileps\otimes_{\epsilon}\left(L-\widetilde{L}\right)(\ptil-q)\right|(t_i,t_j,x,y) &\le C_0(1+K\sqrt T)^{S}(1+\sqrt{T})\,\frac{M_T}{\Gamma\!\left(1+\frac{\gamma_2}{2}\right)} \\
&\qquad\times(t_j-t_i)^{\gamma_2/2}\,\Deln\,Q_{S-d-5}\!\left(\frac{x-y}{\sqrt{t_j-t_i}}\right),
\end{aligned}
\end{equation*}
and $(t_j-t_i)^{\gamma_2/2}\le T^{\gamma_2/2}$.

\emph{Local limit theorem part.} By \eqref{eq:lltpart} with $|\nu|=1,2$, \textbf{A3} and \eqref{eq:smoothing},
\begin{equation*}
\begin{aligned}
&\left|(L_{t_k}-\widetilde{L}_{t_k})(q-\ptileps)\right|(t_k,t_j,z,y)\\
&\qquad\le C_0\,(1+K\sqrt T)^{S}\,h^{1/2}\left[(t_j-t_k)^{-1/2}+(t_j-t_k)^{\frac{\gamma_2-3}{2}}\right]Q_{S-d-4}\!\left(\frac{z-y}{\sqrt{t_j-t_k}}\right).
\end{aligned}
\end{equation*}
The exponent $\frac{\gamma_2-3}{2}<-1$ is not integrable in time, so we perform the discrete summation directly. By \eqref{eq:chapman}, with $m':=j-k$ and $t_j-t_k=m'h$,
\begin{equation*}
\sum_{k=i}^{j-1}h\,h^{1/2}(t_j-t_k)^{\frac{\gamma_2-3}{2}}
= h^{3/2}\sum_{m'=1}^{j-i}\left(m'h\right)^{\frac{\gamma_2-3}{2}}
= h^{\gamma_2/2}\sum_{m'=1}^{j-i}(m')^{\frac{\gamma_2-3}{2}} \le C_0\,h^{\gamma_2/2}\ln\left(e(j-i)\right),
\end{equation*}
since $\sum_{m'\ge1}(m')^{(\gamma_2-3)/2}$ converges for $\gamma_2<1$ and grows logarithmically for $\gamma_2=1$; this part does not depend on $T$. The contribution of the exponent $-1/2$ is
\begin{equation*}
\sum_{k=i}^{j-1}h\,h^{1/2}(t_j-t_k)^{-1/2} = h\sum_{m'=1}^{j-i}(m')^{-1/2} \le 2h\sqrt{j-i} = 2h^{1/2}(t_j-t_i)^{1/2} \le 2\sqrt{T}\,\Deln;
\end{equation*}
Using $h^{\gamma_2/2}\le\Deln$, $\ln(e(j-i))\ge1$ and Lemma~\ref{lem:frozen} for $\ptileps$,
\begin{equation*}
\left|\ptileps\otimes_{\epsilon}(\Htil-H)\right|(t_i,t_j,x,y) \le B_T\ln\left(e(j-i)\right)\Deln\,Q_{S-d-5}\!\left(\frac{x-y}{\sqrt{t_j-t_i}}\right).
\end{equation*}
The higher-order terms are treated by the iterative identity of the proof of Lemma~6 of \cite{bitter2025},
\begin{equation*}
\ptileps\otimes_{\epsilon}\Htil^{r+1} - \ptileps\otimes_{\epsilon}H^{r+1}
= \left(\ptileps\otimes_{\epsilon}\left(\Htil^r-H^r\right)\right)\otimes_{\epsilon}\Htil + \left(\ptileps\otimes_{\epsilon}H^r\right)\otimes_{\epsilon}\left(\Htil-H\right),
\end{equation*}
together with \eqref{eq:iterated} and the same sum--integral comparisons: each additional convolution with $\Htil$ contributes a factor of order $M_T\,\Gamma\!\left(1+\frac{r\gamma_2}{2}\right)/\Gamma\!\left(1+\frac{(r+1)\gamma_2}{2}\right)(t_j-t_i)^{\gamma_2/2}$, cf.\ \eqref{eq:betaconv}. More precisely, unrolling the identity, $\ptileps\otimes_{\epsilon}(\Htil^r-H^r)$ is a sum of $r$ terms $\big((\ptileps\otimes_{\epsilon}H^s)\otimes_{\epsilon}(\Htil-H)\big)\otimes_{\epsilon}\Htil^{r-1-s}$, $0\le s\le r-1$; by \eqref{eq:iterated}, the first-order estimate above (with $(t_k-t_i)^{s\gamma_2/2}\le(t_j-t_i)^{s\gamma_2/2}$) and \eqref{eq:betaconv}, each of them is bounded by
\begin{equation*}
B_T\,\frac{M_T^{r-1}(t_j-t_i)^{\frac{(r-1)\gamma_2}{2}}}{\Gamma\!\left(1+\frac{(r-1)\gamma_2}{2}\right)}\ln\left(e(j-i)\right)\Deln\,Q_{S-d-5}\!\left(\frac{x-y}{\sqrt{t_j-t_i}}\right).
\end{equation*}
Summing over $r\ge1$ with the second relation in \eqref{eq:ML} ($\alpha=\gamma_2/2$, $\beta=1$, $z=M_TT^{\gamma_2/2}$) gives the factor $\mathcal{E}_{\gamma_2/2,1}(2M_TT^{\gamma_2/2})$. This proves \eqref{eq:ibbound}.
\end{proof}

\subsection{The term $i_a$: comparison of $H^{\epsilon}$ with $\Htil$}

\begin{lemma}\label{lem:ia}
For all $0\le t_i<t_j\le T$ and $x,y\in\R^d$,
\begin{equation}\label{eq:iabound}
i_a = \Bigg|\sum_{r=0}^{j-i}\ptileps\otimes_{\epsilon}\left(H^{\epsilon,r}-\Htil^r\right)\Bigg|(t_i,t_j,x,y) \le C^{(i_a)}_T\ln\left(e(j-i)\right)\Deln\,Q_{S-d-6}\!\left(\frac{x-y}{\sqrt{t_j-t_i}}\right),
\end{equation}
with $C^{(i_a)}_T:=C_0(1+K\sqrt T)^{S}A_T\,\mathcal{E}_{\gamma_2/2,1}\!\left(2M_T\,T^{\gamma_2/2}\right)$ and $A_T:=C_0(1+K\sqrt T)^{S}\,C_{H,T}(1+\sqrt T)$.
\end{lemma}

\begin{proof}
\emph{The time shift in $H^{\epsilon}$.} Recall that $H^{\epsilon}(t_i,t_j,\cdot,y)$ is built from $\varphi:=\ptileps(t_{i+1},t_j,\cdot,y)$, while $\Htil(t_i,t_j,\cdot,y)$ is built from $\ptileps(t_i,t_j,\cdot,y)$; we write $\tau:=t_j-t_i$.

\emph{Last step, $j=i+1$.} Here 

\begin{equation*}
\begin{aligned}
p_{\epsilon}(t_i,t_{i+1},x,y)=h^{-d/2}q^{\epsilon,n}_{t_i,x}\big((y-x-hb_{\epsilon}(t_i,x))/\sqrt h\big)
\end{aligned}
\end{equation*}

and 
\begin{equation*}
\begin{aligned}
\ptileps(t_i,t_{i+1},x,y)=h^{-d/2}q^{\epsilon,n}_{t_i,y}\big((y-x-hb_{\epsilon}(t_i,y))/\sqrt h\big)
\end{aligned}
\end{equation*}
By \textbf{A4} (both conditions), \textbf{A3} for $b_{\epsilon}$ and $\sqrt h\,|x-y|=h\,|x-y|/\sqrt h$, their difference is at most $C_0\,h\,(1+K)\,h^{-d/2}(1+|x-y|/\sqrt h)\,Q_S$ evaluated at a point within distance $K\sqrt h$ of $(y-x)/\sqrt h$, hence $|H^{\epsilon}(t_i,t_{i+1},x,y)|\le C_0(1+K\sqrt T)^{S}\,Q_{S-1}((x-y)/\sqrt h)$. Together with \eqref{eq:kernelest} for $\Htil$ and $h\le h^{\gamma_2/2}\le\Deln$,
\begin{equation*}
\begin{aligned}
|H^{\epsilon}-\Htil|(t_i,t_{i+1},x,y)&\le C_0C_{H,T}\left(1+h^{\frac{\gamma_2}2-1}\right)Q_{S-d-4}\!\left(\frac{x-y}{\sqrt h}\right)\\
&\le\frac{2C_0C_{H,T}\Deln}{h}\,Q_{S-d-4}\!\left(\frac{x-y}{\sqrt h}\right).
\end{aligned}
\end{equation*}

\emph{Many steps, $j\ge i+2$.} Then $\tau':=t_j-t_{i+1}=\tau-h\in[\tau/2,\tau]$, so Lemma~\ref{lem:frozen} applies to $\varphi$ with $\tau'$ in place of $\tau$, and all bounds below hold with $\tau$ at the cost of a factor $2^{S}$ in $C_0$. We write
\begin{equation*}
H^{\epsilon}-\Htil=\Big[(L^{\epsilon}_{t_i}-\widetilde L^{\epsilon}_{t_i})\varphi-(L_{t_i}-\widetilde L_{t_i})\varphi\Big]+S_5,\qquad S_5:=(L_{t_i}-\widetilde L_{t_i})\big(\varphi-\ptileps(t_i,t_j,\cdot,y)\big).
\end{equation*}
In the bracket we apply the third-order Taylor expansion of the discrete generators to $\varphi$ as in the proof of Lemma~7 of \cite{bitter2025} (see (3.13) there); this yields $S_1+S_2+S_3+S_4$ below, where $\ptileps$ stands for $\varphi$. Under the freezing at $y$ these terms take the form
\begin{align*}
S_1 &= \sum_{m'=1}^d\left[\left((b_{\epsilon})_{m'}-b_{m'}\right)(t_i,x)-\left((b_{\epsilon})_{m'}-b_{m'}\right)(t_i,y)\right]\frac{\partial\ptileps}{\partial x_{m'}}, \\
S_2 &= \frac12\sum_{m',l=1}^d\left[\left(a^{\epsilon}_{m'l}-a_{m'l}\right)(t_i,x)-\left(a^{\epsilon}_{m'l}-a_{m'l}\right)(t_i,y)\right]\frac{\partial^2\ptileps}{\partial x_{m'}\partial x_l}, \\
S_3 &= \frac{h}{2}\sum_{m',l=1}^d\left[(b_{\epsilon})_{m'}(b_{\epsilon})_l(t_i,x)-(b_{\epsilon})_{m'}(b_{\epsilon})_l(t_i,y)\right]\frac{\partial^2\ptileps}{\partial x_{m'}\partial x_l},
\end{align*}
and $S_4$ is the third-order remainder of \cite{bitter2025}, p.~17, containing the innovation densities $\rho_1=q^{\epsilon,n}_{t_i,x}$, $\rho_2=q^{\epsilon,n}_{t_i,y}$. (In \cite{bitter2025} the decomposition additionally contained the differences of the coefficients at the two freezing points $\theta^{\epsilon}_{t_i,t_j}(y)$ and $\theta_{t_i,t_j}(y)$, with the weight $(1+|y|)$; since $H^{\epsilon}$ and $\Htil$ freeze the coefficients at the same point $y$, these terms vanish identically here.)

\emph{Bound for $S_1$.} By the definition of $\Delta_b$,
\begin{equation*}
\left|\left((b_{\epsilon})_{m'}-b_{m'}\right)(t_i,x)\right| + \left|\left((b_{\epsilon})_{m'}-b_{m'}\right)(t_i,y)\right| \le 2\Delta_b,
\end{equation*}
so by Lemma~\ref{lem:frozen},
\begin{equation*}
|S_1| \le \frac{C_0(1+K\sqrt T)^{S}\,\Delta_b}{\sqrt{t_j-t_i}}\,Q_{S-d-2}\!\left(\frac{x-y}{\sqrt{t_j-t_i}}\right) \le \frac{C_0(1+K\sqrt T)^{S}\,\Deln}{\sqrt{t_j-t_i}}\,Q_{S-d-2}\!\left(\frac{x-y}{\sqrt{t_j-t_i}}\right).
\end{equation*}

\emph{Bound for $S_2$.} Analogously, by the definition of $\Delta_a$ and Lemma~\ref{lem:frozen},
\begin{equation*}
|S_2| \le \frac{C_0(1+K\sqrt T)^{S}\,\Delta_a}{t_j-t_i}\,Q_{S-d-3}\!\left(\frac{x-y}{\sqrt{t_j-t_i}}\right) \le \frac{C_0(1+K\sqrt T)^{S}\,\Deln}{t_j-t_i}\,Q_{S-d-3}\!\left(\frac{x-y}{\sqrt{t_j-t_i}}\right).
\end{equation*}

\emph{Bound for $S_3$.} Under \textbf{A1} the products $(b_{\epsilon})_{m'}(b_{\epsilon})_l$ are bounded by $K_2^2$, so using Lemma~\ref{lem:frozen} and $h\le h^{1/2}\le\Deln$,
\begin{equation*}
|S_3| \le \frac{C_0(1+K\sqrt T)^{S}\,h}{t_j-t_i}\,Q_{S-d-3}\!\left(\frac{x-y}{\sqrt{t_j-t_i}}\right) \le \frac{C_0(1+K\sqrt T)^{S}\,\Deln}{t_j-t_i}\,Q_{S-d-3}\!\left(\frac{x-y}{\sqrt{t_j-t_i}}\right).
\end{equation*}

\emph{Bound for $S_4$.} With $\varphi$ as above write $\Delta_x:=h\,b_{\epsilon}(t_i,x)+\sqrt h\,\xi$, where $\xi$ has density $\rho_1:=q^{\epsilon,n}_{t_i,x}$ (for the frozen generator: $b_{\epsilon}(t_i,y)$ and $\rho_2:=q^{\epsilon,n}_{t_i,y}$). By Taylor's formula with integral remainder,
\begin{equation*}
\varphi(x+\Delta)-\varphi(x)=\langle D\varphi(x),\Delta\rangle+\frac12\langle D^2\varphi(x)\Delta,\Delta\rangle+\sum_{|\eta|=3}\frac{3}{\eta!}\,\Delta^{\eta}\int_0^1(1-\delta)^2D^{\eta}\varphi(x+\delta\Delta)\,d\delta .
\end{equation*}
Since $\xi$ is centered with covariance $a^{\epsilon}$, $\E\Delta_x=hb_{\epsilon}(t_i,x)$ and $\E\Delta_x\Delta_x^{\top}=h^2b_{\epsilon}b_{\epsilon}^{\top}(t_i,x)+h\,a^{\epsilon}(t_i,x)$. Dividing by $h$ and subtracting $\Htil$ gives the decomposition $S_1+S_2+S_3+S_4$ above, with $\Delta^{\eta}=h^{3/2}(\sqrt h\,b_{\epsilon}+\xi)^{\eta}$, i.e.
\begin{equation*}
S_4=3\sqrt h\sum_{|\eta|=3}\frac1{\eta!}\int_{\R^d}\left[\omega_1(z)\varepsilon_1(z)\rho_1(z)-\omega_2(z)\varepsilon_2(z)\rho_2(z)\right]dz,
\end{equation*}
\begin{equation*}
\varepsilon_{1}(z)=\left(\sqrt h\,b_{\epsilon}(t_i,x)+z\right)^{\eta},\quad
\omega_{1}(z)=\int_0^1(1-\delta)^2D^{\eta}\varphi\left(x+\delta h\,b_{\epsilon}(t_i,x)+\delta\sqrt h\,z\right)d\delta,
\end{equation*}
and $\varepsilon_2,\omega_2$ defined with $b_{\epsilon}(t_i,y)$ in place of $b_{\epsilon}(t_i,x)$. We split
\begin{equation*}
|\omega_1\varepsilon_1\rho_1-\omega_2\varepsilon_2\rho_2|\le|(\omega_1-\omega_2)\varepsilon_1\rho_1|+|\omega_2(\varepsilon_1-\varepsilon_2)\rho_1|+|\omega_2\varepsilon_2(\rho_1-\rho_2)|=:1)+2)+3).
\end{equation*}
\emph{Kernel at shifted points.} For $u=\delta h\,b+\delta\sqrt h\,z$ with $|b|\le K$, $\delta\in[0,1]$, we have $|\delta hb|\le K\tau$ and $|\delta\sqrt h z|/\sqrt\tau\le|z|$; by \eqref{eq:shiftkernel} and $1+|w|\le(1+|w+v|)(1+|v|)$,
\begin{equation*}
Q_{S'}\!\left(\frac{x+u-y}{\sqrt\tau}\right)\le(1+K\sqrt T)^{S}\,(1+|z|)^{S'}\,Q_{S'}\!\left(\frac{x-y}{\sqrt\tau}\right),\qquad S'\le S .
\end{equation*}
Hence, by Lemma~\ref{lem:frozen}, $|\omega_2(z)|\le C_0(1+K\sqrt T)^{2S}\tau^{-3/2}(1+|z|)^{S-d-4}Q_{S-d-4}((x-y)/\sqrt\tau)$, and $|\varepsilon_{1,2}(z)|\le C_0(1+|z|)^3$ since $h\le1$.

\emph{Part 3).} By \textbf{A4}, $|\rho_1(z)-\rho_2(z)|\le C_q\sqrt h\,|x-y|\,Q_S(z)$. Since $\int(1+|z|)^{3+S-d-4}Q_S(z)\,dz=c_S\int(1+|z|)^{-d-1}dz<\infty$,
\begin{equation*}
\begin{aligned}
3\sqrt h\int 3)\,dz&\le C_0(1+K\sqrt T)^{2S}\,h\,\frac{|x-y|}{\tau^{3/2}}\,Q_{S-d-4}\!\left(\frac{x-y}{\sqrt\tau}\right)\\
&\le C_0(1+K\sqrt T)^{2S}\,\frac{h}{\tau}\,Q_{S-d-5}\!\left(\frac{x-y}{\sqrt\tau}\right),
\end{aligned}
\end{equation*}
by \eqref{eq:smoothing}.

\emph{Part 2).} $|\varepsilon_1(z)-\varepsilon_2(z)|\le C_0\sqrt h\,|b_{\epsilon}(t_i,x)-b_{\epsilon}(t_i,y)|(1+|z|)^2\le C_0B\sqrt h\,|x-y|(1+|z|)^2$ by \textbf{A3}, and $\rho_1\le C_qQ_S$; exactly as for 3), $3\sqrt h\int2)\,dz\le C_0(1+K\sqrt T)^{2S}\,\frac h\tau\,Q_{S-d-5}((x-y)/\sqrt\tau)$.

\emph{Part 1).} By the mean value theorem, $|\omega_1(z)-\omega_2(z)|\le h\,|b_{\epsilon}(t_i,x)-b_{\epsilon}(t_i,y)|\sup|D^{\eta+e_k}\varphi|$ over the segment, and Lemma~\ref{lem:frozen} with four derivatives gives $|\omega_1-\omega_2|(z)\le C_0(1+K\sqrt T)^{2S}B\,h\,|x-y|\,\tau^{-2}(1+|z|)^{S-d-5}Q_{S-d-5}((x-y)/\sqrt\tau)$. Since $\int(1+|z|)^{3+S-d-5}Q_S<\infty$ and $h^{1/2}\tau^{-1/2}\le1$,
\begin{equation*}
\begin{aligned}
3\sqrt h\int1)\,dz&\le C_0(1+K\sqrt T)^{2S}\,\frac{h^{3/2}}{\tau^{3/2}}\,\frac{|x-y|}{\sqrt\tau}\,Q_{S-d-5}\!\left(\frac{x-y}{\sqrt\tau}\right)\\
&\le C_0(1+K\sqrt T)^{2S}\,\frac h\tau\,Q_{S-d-6}\!\left(\frac{x-y}{\sqrt\tau}\right).
\end{aligned}
\end{equation*}
Altogether $|S_4|\le C_0(1+K\sqrt T)^{2S}\,h\,\tau^{-1}Q_{S-d-6}\le C_0(1+K\sqrt T)^{2S}\Deln\,\tau^{-1}Q_{S-d-6}$, because $h\le h^{1/2}\le\Deln$. (In \cite{bitter2025} the corresponding step uses $\sqrt n\,(t_j-t_i)^{1/2}\ge1$; here the bound $h\le\Deln$ suffices.)

\emph{Bound for $S_5$.} Since the frozen chain is Markov, $\ptileps(t_i,t_j,x,y)=\E\,\varphi(x+\widetilde\Delta)$ with $\widetilde\Delta=hb_{\epsilon}(t_i,y)+\sqrt h\,\widetilde\xi$, $\widetilde\xi\sim q^{\epsilon,n}_{t_i,y}$, which does not depend on $x$. Hence $\varphi-\ptileps(t_i,t_j,\cdot,y)=-h\,\widetilde L^{\epsilon}_{t_i}\varphi$, and $D^{\nu}_x$ commutes with $\widetilde L^{\epsilon}_{t_i}$. By the second-order Taylor formula with integral remainder, $h\widetilde L^{\epsilon}_{t_i}D^{\nu}\varphi(x)=h\langle b_{\epsilon}(t_i,y),D^{\nu+1}\varphi(x)\rangle+\E\int_0^1(1-\delta)\langle D^{2}D^{\nu}\varphi(x+\delta\widetilde\Delta)\widetilde\Delta,\widetilde\Delta\rangle\,d\delta$, with $|\widetilde\Delta|^2\le2(K^2h^2+h|\widetilde\xi|^2)$. Lemma~\ref{lem:frozen} at the shifted points, the kernel shift bound used for $S_4$ (which costs $(1+K\sqrt T)^{S}(1+|\widetilde\xi|)^{S'}$), and the finiteness of the moments of $\widetilde\xi$ of order $<S-d$ give, for $|\nu|\le2$,
\begin{equation*}
\left|D^{\nu}_x\big(\varphi-\ptileps(t_i,t_j,\cdot,y)\big)\right|\le C_0(1+K\sqrt T)^{2S}(1+\sqrt T)\,\frac{h}{\tau^{1+|\nu|/2}}\,Q_{S-d-3-|\nu|}\!\left(\frac{x-y}{\sqrt\tau}\right),
\end{equation*}
where the drift term $h\tau^{-1/2-|\nu|/2}$ was bounded by $\sqrt T\,h\tau^{-1-|\nu|/2}$. Using $|a(t_i,x)-a(t_i,y)|\le C_0|x-y|^{\gamma_2}$ with \eqref{eq:smoothing} and $|b(t_i,x)-b(t_i,y)|\le2K$ exactly as in \eqref{eq:kernelest}, we get
\begin{equation*}
\begin{aligned}
|S_5|&\le C_0(1+K\sqrt T)^{S}C_{H,T}(1+\sqrt T)\,\frac{h}{\tau^{2-\frac{\gamma_2}{2}}}\,Q_{S-d-6}\!\left(\frac{x-y}{\sqrt\tau}\right)\\
&\le C_0(1+K\sqrt T)^{S}C_{H,T}(1+\sqrt T)\,\frac{\Deln}{\tau}\,Q_{S-d-6}\!\left(\frac{x-y}{\sqrt\tau}\right),
\end{aligned}
\end{equation*}
since $h\,\tau^{\gamma_2/2-1}\le h^{\gamma_2/2}\le\Deln$ for $\tau\ge h$.

Combining with the bounds for $S_1$ (where $\tau^{-1/2}\le\sqrt T\,\tau^{-1}$), $S_2$, $S_3$, $S_5$ and with the last step,
\begin{equation*}
|H^{\epsilon}-\Htil|(t_i,t_j,x,y) \le \frac{A_T\,\Deln}{t_j-t_i}\,Q_{S-d-6}\!\left(\frac{x-y}{\sqrt{t_j-t_i}}\right),\qquad A_T=C_0(1+K\sqrt T)^{S}\,C_{H,T}(1+\sqrt T),
\end{equation*}

\emph{Convolutions.} By Lemma~\ref{lem:frozen} and \eqref{eq:chapman},
\begin{equation*}
\begin{aligned}
\left|\ptileps\otimes_{\epsilon}(H^{\epsilon}-\Htil)\right|(t_i,t_j,x,y) &\le C_0(1+K\sqrt T)^{S}A_T\,\Deln\,Q_{S-d-6}\!\left(\frac{x-y}{\sqrt{t_j-t_i}}\right)\sum_{k=i}^{j-1}\frac{h}{t_j-t_k} \\
&\le C_0(1+K\sqrt T)^{S}A_T\ln\left(e(j-i)\right)\Deln\,Q_{S-d-6}\!\left(\frac{x-y}{\sqrt{t_j-t_i}}\right),
\end{aligned}
\end{equation*}
since $\sum_{k=i}^{j-1}\frac{h}{t_j-t_k}=\sum_{m'=1}^{j-i}\frac{1}{m'}\le\ln\left(e(j-i)\right)$; this sum depends neither on $h$ nor on $T$. For the higher-order terms we use the algebraic identity
\begin{equation*}
H^{\epsilon,r}-\Htil^r=\sum_{s=0}^{r-1}\Htil^{s}\otimes_{\epsilon}\left(H^{\epsilon}-\Htil\right)\otimes_{\epsilon}H^{\epsilon,r-1-s}
\end{equation*}
(discrete convolutions; it follows by induction from $H^{\epsilon,r}-\Htil^r=(H^{\epsilon,r-1}-\Htil^{r-1})\otimes_{\epsilon}H^{\epsilon}+\Htil^{r-1}\otimes_{\epsilon}(H^{\epsilon}-\Htil)$). For each piece, \eqref{eq:iterated} bounds $\ptileps\otimes_{\epsilon}\Htil^{s}$ by $C_0(1+K\sqrt T)^{S}M_T^s(t_k-t_i)^{s\gamma_2/2}/\Gamma(1+\frac{s\gamma_2}2)\,Q_{S-d-4}$; the convolution with $H^{\epsilon}-\Htil$ is estimated as in the case $r=1$, using $\sum_k h\frac{(t_k-t_i)^{s\gamma_2/2}}{t_j-t_k}\le(t_j-t_i)^{s\gamma_2/2}\ln(e(j-i))$, which gives
\begin{equation*}
\left|\left(\ptileps\otimes_{\epsilon}\Htil^s\right)\otimes_{\epsilon}\left(H^{\epsilon}-\Htil\right)\right|(t_i,t_k) \le C_0(1+K\sqrt T)^{S}A_T\,\frac{M_T^s\,(t_k-t_i)^{\frac{s\gamma_2}{2}}}{\Gamma\!\left(1+\frac{s\gamma_2}{2}\right)}\ln\left(e(k-i)\right)\Deln\,Q_{S-d-6};
\end{equation*}
finally, since $\ln(e(k-i))\le\ln(e(j-i))$, \eqref{eq:betaconv} with $a=s\gamma_2/2$ for the kernel $H^{\epsilon}$ (whose bound is \eqref{eq:kernelest}) shows that each of the $r$ pieces is bounded by $C_0(1+K\sqrt T)^{S}A_T M_T^{r-1}(t_j-t_i)^{(r-1)\gamma_2/2}/\Gamma(1+\frac{(r-1)\gamma_2}{2})\cdot\ln(e(j-i))\Deln\,Q_{S-d-6}$. This completes the estimate of $i_a$; the sum over $r$ is bounded by means of \eqref{eq:ML}, which gives the factor $\mathcal{E}_{\gamma_2/2,1}(2M_TT^{\gamma_2/2})$.
\end{proof}

\subsection{Conclusion of the proof of Proposition~\ref{thm:density}}

By \eqref{eq:decomposition}, the split form of \eqref{eq:r0}, \eqref{eq:iibound}, \eqref{eq:iiibound}, \eqref{eq:ibbound} and \eqref{eq:iabound}, using $1\le T^{\frac{1-\gamma_2}{2}}(t_j-t_i)^{\frac{\gamma_2-1}{2}}$, $\ln(e(j-i))\ge1$ and $Q_{S-d-k}\le Q_{S-d-6}$ for $k\le6$, we obtain for all $0\le t_i<t_j\le T$ and $x,y\in\R^d$
\begin{equation}\label{eq:finalT}
|p-p_{\epsilon}|(t_i,t_j,x,y) \le C_4(T)\left[\ln\left(e(j-i)\right)(t_j-t_i)^{\frac{\gamma_2-1}{2}}\Deln+(j-i)^{-1/2}\right]Q_{S-d-6}\!\left(\frac{x-y}{\sqrt{t_j-t_i}}\right),
\end{equation}
where
\begin{equation*}
C_4(T) := C_0(1+K\sqrt T)^{S}\left[1+(1+\sqrt T)\,T^{\frac{1-\gamma_2}{2}}\right]+C^{(ii)}_T+T^{\frac{1-\gamma_2}{2}}\left(C^{(iii)}_T+C^{(i_a)}_T+C^{(i_b)}_T\right),
\end{equation*}
with $C^{(ii)}_T$, $C^{(iii)}_T$, $C^{(i_a)}_T$, $C^{(i_b)}_T$ from Lemmas~\ref{lem:ii}, \ref{lem:iii}, \ref{lem:ia}, \ref{lem:ib}. Note that $\ln(e(j-i))\le\ln(en)$, since $j-i\le n$.

This completes the proof of Proposition~\ref{thm:density}.

\section*{Proof of Propositions \ref{thm:grid} and \ref{thm:sup}}

\paragraph{Dependence of the constant $C_4$ on $T$.} Proposition~\ref{thm:grid} rests on the density comparison of Proposition~\ref{thm:density}, which is the analogue of \cite[Theorem~1]{bitter2025} for an arbitrary horizon. In the proof of Proposition~\ref{thm:density} the horizon enters only through four factors:
(i) the shift factor $(1+K\sqrt T)^S$ of \eqref{eq:shiftkernel}, since the drift moves the centre of the frozen densities by at most $K(t_j-t_i)\le K\sqrt T\sqrt{t_j-t_i}$;
(ii) the factor $(1+T)^{1-\gamma_2/2}$ in $C_{H,T}$, which comes from writing the drift part $(s-t)^{-1/2}$ of the kernels as $T^{\frac{1-\gamma_2}{2}}(s-t)^{\gamma_2/2-1}$ and the terms of order $h/\tau$ in $H^{\epsilon}$ as $T^{1-\gamma_2}\tau^{\gamma_2/2-1}$ in \eqref{eq:kernelest};
(iii) the factor $1+\sqrt T$ coming from the difference of the means in Lemma~\ref{lem:llt};
(iv) the Mittag-Leffler functions $\mathcal{E}_{\gamma_2/2,\cdot}(M_TT^{\gamma_2/2})$, $M_T=C_0C_{H,T}\Gamma(\gamma_2/2)$, which arise when the parametrix series are summed.
All other constants (Theorem~19.3 of \cite{bhattacharya1976}, \eqref{eq:smoothing}, \eqref{eq:chapman} and the sum--integral comparisons) do not depend on $T$. Collecting (i)--(iv) gives $C_4(T)$ in \eqref{eq:finalT}; for $T=1$ it is a constant of the same form as in \cite{bitter2025}.

We first derive the $L^1$-distance between the transition densities for a single step of the macro-grid. Recall that adjacent macro-points $\tau_{i-1}<\tau_i$ are separated by $n/m$ micro-steps and by the time $\tau_i-\tau_{i-1}=T/m$.

\begin{lemma}\label{lem:l1}
There exists a constant $C(T)>0$ (namely $C(T)=C_4(T)$ from Proposition~\ref{thm:density}) such that for any adjacent macro-points $\tau_{i-1}<\tau_i$ and $x\in\R^d$:
\begin{equation*}
\int_{\R^d}|p(\tau_{i-1},\tau_i,x,y)-p_{\epsilon}(\tau_{i-1},\tau_i,x,y)|\,dy \le C\,\Delnm,
\end{equation*}
where
\begin{equation}\label{eq:deltanm}
\Delnm := \ln\left(\frac{en}{m}\right)\left(\frac{T}{m}\right)^{\frac{\gamma_2-1}{2}}\Deln+\left(\frac{m}{n}\right)^{1/2}
\end{equation}
is the expression in square brackets in Proposition~\ref{thm:grid} and Theorem~\ref{thm:chaindif} (recall $\Deln=(T/n)^{\varrho}+\Delta_a+\Delta_b$, $\varrho=\min(\gamma_2/2,\alpha_2,\beta)$).
\end{lemma}

\begin{proof}
Integrating the right-hand side of the bound in Proposition~\ref{thm:density} with respect to $y$, we apply the change of variables $z=y-x$ and the rescaling $w=z/\sqrt{\tau_i-\tau_{i-1}}$. The Jacobian of the latter transformation is $(\tau_i-\tau_{i-1})^{d/2}$. Consequently:
\begin{equation*}
\int_{\R^d}\widehat{Q}_{S-d-6}\!\left(\frac{y-x}{\sqrt{\tau_i-\tau_{i-1}}}\right)dy = \int_{\R^d}Q_{S-d-6}(w)\,dw = 1.
\end{equation*}
The result follows from \eqref{eq:densitycomp} with $j-i=n/m$ and $t_j-t_i=T/m$.
\end{proof}

\begin{remark}\label{rem:correction}
Let $n=m^l$ with $l>1$ and $m\ge3$. Then $h=Tm^{-l}$, $\ln(en/m)=1+(l-1)\ln m\le l\ln m$, and \eqref{eq:deltan}, \eqref{eq:deltanm} give
\begin{equation*}
m\Delnm \le T^{\frac{\gamma_2-1}{2}}\,l\ln m\left[T^{\varrho}\,m^{\frac{3-\gamma_2}{2}-l\varrho}+m^{\frac{3-\gamma_2}{2}}(\Delta_a+\Delta_b)\right]+m^{\frac{3-l}{2}}, \qquad \varrho=\min(\gamma_2/2,\alpha_2,\beta).
\end{equation*}
Hence the $n$-dependent part of $m\Delnm$ tends to zero as $m\to\infty$ if and only if $l>\frac{3-\gamma_2}{2\varrho}$ and $l>3$. 
\end{remark}

Let $0=\tau_0<\tau_1<\dots<\tau_m=T$ be the fixed macro-grid. We denote by $P_X^m(z_1,\dots,z_m)$ and $P_Y^m(z_1,\dots,z_m)$ the probability densities of the vectors $(X_{\tau_1},\dots,X_{\tau_m})$ (the diffusion process) and $(Y_{\tau_1},\dots,Y_{\tau_m})$ (the Markov chain), respectively, conditioned on the common starting point $X_0=Y_0=x$.

\begin{lemma}\label{lem:telescope}
The following estimate holds for the densities on the macro-grid:
\begin{equation*}
\int_{\R^{m\times d}}|P_X^m(z)-P_Y^m(z)|\,dz \le Cm\Delnm,
\end{equation*}
where $z=(z_1,\dots,z_m)$ and $dz=\de Z_1\dots \de Z_m$.
\end{lemma}

\begin{proof}
Both processes are Markovian, so the densities factorize into products of one-step transition densities on the macro-grid. We employ a telescopic decomposition for the product of densities:
\begin{align*}
|P_X^m(z)-P_Y^m(z)| &= \Bigg|\prod_{i=1}^m p(\tau_{i-1},\tau_i,z_{i-1},z_i) - \prod_{i=1}^m p_{\epsilon}(\tau_{i-1},\tau_i,z_{i-1},z_i)\Bigg| \\
&\le \sum_{i=1}^m\left(\prod_{k=1}^{i-1}p_{\epsilon}(\tau_{k-1},\tau_k,z_{k-1},z_k)\right)|p(\tau_{i-1},\tau_i,z_{i-1},z_i)-p_{\epsilon}(\tau_{i-1},\tau_i,z_{i-1},z_i)| \\
&\quad\times\left(\prod_{l=i+1}^m p(\tau_{l-1},\tau_l,z_{l-1},z_l)\right),
\end{align*}
with $z_0=x$.

To evaluate the integral over the full space $dz=\de Z_1\de Z_2\dots \de Z_m$, we integrate each of the $m$ terms in the sum sequentially, starting from the last variable $z_m$ and moving backwards to $z_1$. Let us analyze the integration of an arbitrary fixed $i$-th term from the sum. The multiple integral for this term can be factored as follows:
\begin{align*}
\int_{\R^{(i-1)d}}\underbrace{\left(\prod_{k=1}^{i-1}p_{\epsilon}(\tau_{k-1},\tau_k,z_{k-1},z_k)\right)}_{\text{1st block}}
\times\Bigg(\int_{\R^d}|p(\tau_{i-1},&\tau_i,z_{i-1},z_i)-p_{\epsilon}(\tau_{i-1},\tau_i,z_{i-1},z_i)| \\
\times\underbrace{\left(\int_{\R^{(m-i)d}}\prod_{l=i+1}^m p(\tau_{l-1},\tau_l,z_{l-1},z_l)\,\de Z_m\dots \de Z_{i+1}\right)}_{\text{2nd block}}&\,\de Z_i\Bigg)\,\de Z_{i-1}\dots \de Z_1.
\end{align*}
We evaluate the 2nd block by applying sequential integration from variable $z_m$ to $z_{i+1}$. Since each $p(\tau_{l-1},\tau_l,z_{l-1},\cdot)$ is a probability density, its integral over $\R^d$ is equal to 1, and the 2nd block is equal to 1. Now we evaluate the remaining integral:
\begin{equation}\label{eq:Ii}
\begin{aligned}
I_i := \int_{\R^{(i-1)d}}&\left(\prod_{k=1}^{i-1}p_{\epsilon}(\tau_{k-1},\tau_k,z_{k-1},z_k)\right)\\
&\times\left(\int_{\R^d}|p(\tau_{i-1},\tau_i,z_{i-1},z_i)-p_{\epsilon}(\tau_{i-1},\tau_i,z_{i-1},z_i)|\,\de Z_i\right)\de Z_{i-1}\dots \de Z_1.
\end{aligned}
\end{equation}
From Lemma~\ref{lem:l1}, the integral of the difference over the variable $z_i$ is uniformly bounded:
\begin{equation*}
\int_{\R^d}|p(\tau_{i-1},\tau_i,z_{i-1},z_i)-p_{\epsilon}(\tau_{i-1},\tau_i,z_{i-1},z_i)|\,\de Z_i \le C\Delnm.
\end{equation*}
Substituting this bound back into the integral \eqref{eq:Ii}:
\begin{equation*}
I_i \le C\Delnm\int_{\R^{(i-1)d}}\prod_{k=1}^{i-1}p_{\epsilon}(\tau_{k-1},\tau_k,z_{k-1},z_k)\,\de Z_{i-1}\dots \de Z_1 = C\Delnm.
\end{equation*}
The last equality holds because each $p_{\epsilon}(\tau_{k-1},\tau_k,z_{k-1},\cdot)$ is a probability density. Applying this procedure to all $m$ terms, we obtain the final bound:
\begin{equation*}
\int_{\R^{m\times d}}|P_X^m(z)-P_Y^m(z)|\,dz \le \sum_{i=1}^m C\Delnm = Cm\Delnm. \qedhere
\end{equation*}
\end{proof}

We now construct the maximal coupling. Let the vector $V=(X_{\tau_1},\dots,X_{\tau_m})$ be already sampled from the distribution of the diffusion, with density $P(z)=P_X^m(z)$. We wish to generate a new vector $V^{\epsilon}$ distributed according to $Q(z)=P_Y^m(z)$ conditionally on $V$, such that the probability of $V\ne V^{\epsilon}$ is minimized. Define the function $H(z):=\min(P(z),Q(z))$ and the constant $W$:
\begin{equation*}
W := \int_{\R^{m\times d}}H(z)\,dz.
\end{equation*}
From the identity $2\min(P(z),Q(z))=P(z)+Q(z)-|P(z)-Q(z)|$ it follows that:
\begin{equation*}
W = 1-\|P-Q\|_{TV} = 1-\frac12\int_{\R^{m\times d}}|P(z)-Q(z)|\,dz.
\end{equation*}
From Lemma~\ref{lem:telescope}, we have $1-W\le Cm\Delnm$. Given the realized value of $V$, we construct $V^{\epsilon}$ as follows:
\begin{enumerate}
\item With probability $\frac{H(V)}{P(V)}$, we accept the original trajectory and set $V^{\epsilon}=V$.
\item With the probability $1-\frac{H(V)}{P(V)}$, we reject the original trajectory and sample a new vector $V^{\epsilon}$ from the residual density $R(z)$ independently of the vector $V$:
\begin{equation*}
R(z) = \frac{Q(z)-H(z)}{1-W}.
\end{equation*}
\end{enumerate}
By construction, the distribution of $V^{\epsilon}$ is exactly $Q(z)$. Indeed, for any measurable set $A$:
\begin{align*}
\PP(V^{\epsilon}\in A) &= \int_{\R^{m\times d}}\left(\frac{H(v)}{P(v)}\,\mathbf{1}_{\{v\in A\}}+\left(1-\frac{H(v)}{P(v)}\right)\int_A R(z)\,dz\right)P(v)\,dv \\
&= \int_A H(v)\,dv + \int_A R(z)\,dz\int_{\R^{m\times d}}(P(v)-H(v))\,dv \\
&= \int_A H(v)\,dv + \int_A R(z)\,dz\,(1-W) \\
&= \int_A H(v)\,dv + \int_A(Q(v)-H(v))\,dv = \int_A Q(v)\,dv.
\end{align*}
The probability that the chain diverges from the already realized diffusion trajectory is given by the total rejection probability, which is bounded by the $L^1$-distance:
\begin{equation*}
\PP(V\ne V^{\epsilon}) = \int_{\R^{m\times d}}\left(1-\frac{H(v)}{P(v)}\right)P(v)\,dv = 1-W = \frac12\|P-Q\|_{L^1} \le Cm\Delnm.
\end{equation*}
Setting $Y_{\tau_k}$ equal to the $k$-th component of $V^{\epsilon}$, we obtain versions of the diffusion and of the chain on a common probability space with
\begin{equation*}
\PP\left(X_{\tau_k}=Y_{\tau_k} \ \text{for all } k=1,\dots,m\right) \ge 1-Cm\Delnm,
\end{equation*}
which completes the proof of Proposition~\ref{thm:grid}.

Now we can prove Proposition~\ref{thm:sup}.

\begin{proof}[Proof of Proposition \ref{thm:sup}]
\emph{Step 1: fluctuations of the diffusion.} For $s\in[0,\frac{T}{m}]$ consider the increment
\begin{equation*}
\Delta X_{s,i} := X_{\tau_i+s}-X_{\tau_i} = M_{s,i} + \int_{\tau_i}^{\tau_i+s}b(t,X_t)\,\de t,\qquad M_{s,i}:=\int_{\tau_i}^{\tau_i+s}\sigma(t,X_t)\,\de W_t .
\end{equation*}
We use the following inequality.

\begin{lemma}[Exercise 3.16 from \cite{revuzyor1999}]\label{lem:revuz}
Let $M=(M_t)_{t\ge0}$ be a continuous local martingale vanishing at zero. If there exists a constant $c>0$ such that $\langle M\rangle_t\le ct$ for all $t\ge0$, then for any $a>0$,
\begin{equation*}
\PP\left(\sup_{0\le s\le t}M_s\ge at\right)\le\exp\left(-\frac{a^2t}{2c}\right).
\end{equation*}
\end{lemma}

Applying the lemma to $M$ and $-M$ we get $\PP(\sup_{0\le s\le t}|M_s|\ge at)\le2\exp(-\frac{a^2t}{2c})$. Each coordinate $M^{(l)}_{s,i}=\sum_{k=1}^d\int_{\tau_i}^{\tau_i+s}\sigma_{lk}(t,X_t)\,dW^{(k)}_t$ (stopped at $s=T/m$) is a continuous local martingale with
\begin{equation*}
\langle M^{(l)}_{\cdot,i}\rangle_s=\sum_{k=1}^d\int_{\tau_i}^{\tau_i+s}\sigma_{lk}^2(t,X_t)\,\de t\le K^2s
\end{equation*}
by \textbf{A1}. We apply Lemma~\ref{lem:revuz} with $c=K^2$, $t=T/m$ and $a=\frac{m^{\alpha}}{\sqrt d\,T}$, so that $at=\frac{m^{\alpha-1}}{\sqrt d}$ and $\frac{a^2t}{2c}=\frac{m^{2\alpha-1}}{2dK^2T}$:
\begin{equation*}
\PP\left(\sup_{0\le s\le\frac{T}{m}}|M^{(l)}_{s,i}|\ge\frac{m^{\alpha-1}}{\sqrt d}\right)\le2\exp\left(-\frac{m^{2\alpha-1}}{2dK^2T}\right).
\end{equation*}
If $|M_{s,i}|\ge m^{\alpha-1}$, then $|M^{(l)}_{s,i}|\ge m^{\alpha-1}/\sqrt d$ for some coordinate $l$. Since $\int_{\tau_i}^{\tau_{i+1}}|b(t,X_t)|\,\de t\le K_1T/m\le KT/m$, the union bound over the $d$ coordinates and the $m$ cells gives
\begin{equation}\label{eq:Xfluct}
\PP\left(\sup_{0\le i\le m-1}\ \sup_{0\le s\le\frac{T}{m}}|\Delta X_{s,i}|>m^{\alpha-1}+\frac{KT}{m}\right)\le2dm\exp\left(-\frac{m^{2\alpha-1}}{2dK^2T}\right).
\end{equation}

\emph{Step 2: fluctuations of the chain.} By Proposition~\ref{thm:chainfluct} with $\Sigma_l=I$, $\kappa=1$, $r=S$, $C_*=C_q$ (the conditional densities $q^{\epsilon,n}$ of the noise $\sigma_{\epsilon}(t_l,Y_{t_l})\xi^n_{l+1}$ are bounded by $C_qQ_S$ by \textbf{A4}), with $c=C_qc_S\omega_d/(S-d-1)$,
\begin{equation*}
\PP\Big\{\sup_{t\in[\tau_k,\tau_{k+1}]}|Y_t-Y_{\tau_k}|\ge m^{\alpha-1}+\frac{KT}{m}+\frac{c\sqrt{Tn}}{m\,R^{S-d-1}}\text{ for some }k\Big\}\le2dm\exp\left(-\frac{m^{2\alpha-1}}{2dR^2T}\right)+\frac{c\,n}{R^{S-d}},
\end{equation*}

\emph{Step 3.} For $t\in[\tau_k,\tau_{k+1}]$,
\begin{equation*}
|(X_t-Y_t)-(X_{\tau_k}-Y_{\tau_k})|\le|X_t-X_{\tau_k}|+|Y_t-Y_{\tau_k}|,
\end{equation*}
so the event in the statement of Proposition~\ref{thm:sup} is contained in the union of the events of Steps~1 and~2, whose thresholds $m^{\alpha-1}+KTm^{-1}$ and $m^{\alpha-1}+KTm^{-1}+c\sqrt{Tn}\,m^{-1}R^{d+1-S}$ add up to the threshold of Proposition~\ref{thm:sup}. Steps~1 and~2 complete the proof.
\end{proof}

\section*{Proof of Propositions \ref{thm:pert} and \ref{thm:deg}}

\paragraph{Dependence of the constants $C_1$, $C_2$, $C_3$ on $T$.} In \cite{konakov2017stability} and \cite{kozhina2018} the horizon $T$ is fixed, and the constants of the stability bounds depend on $T$ without this dependence being written out (in \cite{konakov2017stability} the letter $C$ denotes, by convention, a constant depending on the assumptions and on $T$). We make it explicit by summing the iterated bounds proved there.

\emph{The constant $C_1$ (diffusions).} By \cite[(2.9)]{konakov2017stability}, the parametrix kernel is bounded by $c_1(1\vee T^{(1-\gamma)/2})(t-u)^{\gamma/2-1}p_c(t-u,z-y)$ with $c_1$ independent of $T$. The one-step bound \cite[(3.7)]{konakov2017stability} contains the two powers $(t-s)^{\gamma/2}$ and $(t-s)^{\alpha(q)}$, $\alpha(q)=\frac{q-d}{2q}$. For $q<\infty$ we bound both by $\max(1,T)^{|\gamma/2-\alpha(q)|}(t-s)^{\vartheta_q}$, $\vartheta_q=\frac\gamma2\wedge\alpha(q)$; in \cite[(3.8)]{konakov2017stability} this factor is omitted, which is harmless only for $T\le1$. For $q=\infty$ only $(t-s)^{\gamma/2}$ appears. Hence, with $\bar C:=c_1\max(1,T)^{a_q}$ (the Beta-function factors of \cite[(3.8)]{konakov2017stability}, which do not depend on $T$, are absorbed in $c_1$), the iterated bound \cite[Lemma~3.4, (3.14)]{konakov2017stability} reads
\begin{equation*}
\big|\ptil\otimes H^{(r)}-\ptileps\otimes H_{\epsilon}^{(r)}\big|(s,t,x,y)\le(r+1)\,\Delta_{\epsilon,\gamma,q}\,\bar C^{\,r+1}\,\frac{\Gamma(\vartheta_q)^r\,(t-s)^{r\vartheta_q}}{\Gamma(1+r\vartheta_q)}\;p_c(t-s,y-x).
\end{equation*}
Summing over $r\ge0$ (parametrix expansions \cite[(2.6), (2.7)]{konakov2017stability}), using $t-s\le T$ and the second inequality in \eqref{eq:ML}, we obtain \eqref{eq:diffstab} with the constant
\begin{equation*}
\bar C\sum_{r\ge0}(r+1)\frac{\big(\bar C\,\Gamma(\vartheta_q)\,T^{\vartheta_q}\big)^r}{\Gamma(1+r\vartheta_q)}\le\bar C\,\mathcal{E}_{\vartheta_q,1}\big(2\bar C\,\Gamma(\vartheta_q)\,T^{\vartheta_q}\big)=C_1(q,T).
\end{equation*}

\emph{The constant $C_3$ (chains).} Under $(I_{P,M})$ the iterated bound for the chains is \cite[(3.40)]{konakov2017stability}, proved for $q=\infty$ only:
\begin{equation*}
\begin{aligned}
&\big|\widetilde p^h\otimes_hH^{h,(r)}-\widetilde p^h_{\epsilon}\otimes_hH_{\epsilon}^{h,(r)}\big|(t_i,t_j,x,y)\\
&\qquad\le(r+1)\,\Delta_{\epsilon,\gamma,\infty}\,\big\{c_3\max(1,T)^{\frac{1-\gamma}{2}}\big\}^{r+1}\,\frac{\Gamma(\frac\gamma2)^r\,(t_j-t_i)^{\frac{r\gamma}{2}}}{\Gamma(1+\frac{r\gamma}2)}\;\chi_c(t_j-t_i,y-x),
\end{aligned}
\end{equation*}
where $c_3$ is the $T$-independent constant denoted $c_1$ in \cite[Section~3.2]{konakov2017stability}. The same summation (now $\vartheta=\gamma/2$ and no additional factor) gives \eqref{eq:chainstab} with $C_3(T)$ of \eqref{eq:C3}. For Gaussian innovations, \cite{konakov2017stability} follows the steps of Lemmas~3.3 and~3.4 for every $q\in(d,\infty]$, which gives $C_3(q,T)$.

\emph{The constant $C_2$ (degenerate diffusions).} By \cite[(4.20)]{kozhina2018} the parametrix kernel is bounded by $c_2(1\vee T^{(1-\gamma)/2})(t-u)^{\gamma/2-1}p_{c,K}$, $\gamma=\gamma_3$, and the constant of \cite[Lemma~4.2.1]{kozhina2018} is $c_2(1\vee T^{(1-\gamma)/2})\vee C_K$ with $C_K$ independent of $T$; enlarging $c_2$ we may take it equal to $c_2\max(1,T)^{(1-\gamma)/2}$. As above, for $q<\infty$ the powers $t^{\gamma/2}$ and $t^{1/2-2d/q}$ of the one-step bound \cite[(4.29)]{kozhina2018} are written as $t^{\vartheta^K_q}$ at the cost of $\max(1,T)^{|\gamma/2-1/2+2d/q|}$. With $\bar C^K:=c_2\max(1,T)^{b_q}$ and $\vartheta:=\vartheta^K_q$, the iterated bound \cite[Lemma~4.3.4, (4.35)]{kozhina2018} reads
\begin{equation*}
\big|\ptil\otimes H^{(r)}-\ptileps\otimes H_{\epsilon}^{(r)}\big|\le(\bar C^K)^r\,r\,\overline{\Delta}_{\epsilon,\gamma,q}\left[\frac{t^{r\vartheta}}{\Gamma(1+r\vartheta)}+\frac{t^{(r+2)\vartheta}}{\Gamma(1+(r+2)\vartheta)}\right]p_{c,K}
\end{equation*}
(in \cite[(4.35)]{kozhina2018} only the exponent $\gamma/2$ appears, which is the case $q=\infty$). Since $\Gamma(1+(r+2)\vartheta)\ge\frac12\Gamma(1+r\vartheta)$ ($\Gamma\ge0.88$ on $[1,\infty)$, $\Gamma\le1$ on $[1,2]$ and $\Gamma$ is increasing on $[2,\infty)$), the bracket is at most $(1+2T^{2\vartheta})\,t^{r\vartheta}/\Gamma(1+r\vartheta)$. Adding the term $r=0$, bounded by $c_2\overline{\Delta}_{\epsilon,\gamma,q}p_{c,K}$ by \cite[Lemma~4.3.2]{kozhina2018}, and using $r\le2^r$, we obtain the single-step bound below with
\begin{equation*}
\bar C^K\big(1+2T^{2\vartheta}\big)\sum_{r\ge0}\frac{\big(2\bar C^KT^{\vartheta}\big)^r}{\Gamma(1+r\vartheta)}=\bar C^K\big(1+2T^{2\vartheta}\big)\,\mathcal{E}_{\vartheta,1}\big(2\bar C^KT^{\vartheta}\big)=C_2(q,T).
\end{equation*}

Propositions~\ref{thm:pert} and~\ref{thm:deg} are proved by the same scheme as Proposition~\ref{thm:grid}: the single-step $L^1$-bound of Lemma~\ref{lem:l1} is replaced by the corresponding known estimates for the perturbed processes, after which the telescoping decomposition of Lemma~\ref{lem:telescope} and the maximal coupling construction apply without changes, with the pair (diffusion, chain) replaced by the pair (original process, perturbed process).

\medskip
\noindent\textbf{Non-degenerate case (Proposition \ref{thm:pert}).} In the case where $\sigma\sigma^*$ is uniformly elliptic, we have the following bound for the difference between the transition densities $p^h$ of the chain \eqref{eq:chain} and $p^h_{\epsilon}$ of the perturbed chain \eqref{eq:chainpert} (\cite[Theorem~1.3]{konakov2017stability}, under $(I_{P,M})$ with $q=\infty$, see above):
\begin{equation}\label{eq:chainstab}
|p^h(t_i,t_j,x,y)-p^h_{\epsilon}(t_i,t_j,x,y)| \le C_3(T)\,\Delta_{\epsilon,\gamma,\infty}\,\chi_c(t_j-t_i,y-x),
\end{equation}
and also for transition densities of diffusions \eqref{eq:sde} and \eqref{eq:sdepert} we have (\cite[Theorem~1.1]{konakov2017stability}, any $q\in(d,\infty]$):
\begin{equation}\label{eq:diffstab}
|p(t_i,t_j,x,y)-p_{\epsilon}(t_i,t_j,x,y)| \le C_1(q,T)\,\Delta_{\epsilon,\gamma,q}\,\rho_c(t_j-t_i,y-x),
\end{equation}
where $\gamma$ is the H\"older continuity exponent of the coefficients, and $C_1(q,T)$, $C_3(T)$ are the explicit constants \eqref{eq:CKKM}, \eqref{eq:C3}. Functions $\chi_c$ and $\rho_c$ are defined as follows (here $Q_r$ denotes the unscaled kernel):
\begin{align*}
\chi_c(u,z) &:= \frac{c^d}{u^{d/2}}\,Q_{M-(d+5+\gamma)}\!\left(\frac{|z|}{u^{1/2}/c}\right), \\
\rho_c(u,z) &:= \frac{c^{d/2}}{(2\pi u)^{d/2}}\exp\left(-c\,\frac{|z|^2}{2u}\right).
\end{align*}
By integrating the inequalities \eqref{eq:chainstab} and \eqref{eq:diffstab} over $y\in\R^d$, we obtain the total variation distance between the transition densities for a single step of the grid:

\begin{lemma}\label{lem:l1pert}
For transition densities $p^h$, $p^h_{\epsilon}$ of Markov chains, for any $0\le t_i<t_j\le T$ and $x\in\R^d$:
\begin{equation*}
\int_{\R^d}|p^h(t_i,t_j,x,y)-p^h_{\epsilon}(t_i,t_j,x,y)|\,dy \le C_3(T)\,\Delta_{\epsilon,\gamma,\infty}.
\end{equation*}
For transition densities $p$, $p_{\epsilon}$ of diffusions, likewise:
\begin{equation*}
\int_{\R^d}|p(t_i,t_j,x,y)-p_{\epsilon}(t_i,t_j,x,y)|\,dy \le C_1(q,T)\,\Delta_{\epsilon,\gamma,q}.
\end{equation*}
\end{lemma}

\begin{proof}
Integrating the right-hand side of \eqref{eq:chainstab} and \eqref{eq:diffstab} with respect to $y$, we apply the change of variables $z=y-x$ and the rescaling $w=\frac{z}{(t_j-t_i)^{1/2}/c}$. The Jacobian of the latter transformation is $(t_j-t_i)^{d/2}c^{-d}$. Consequently:
\begin{equation*}
\begin{aligned}
\int_{\R^d}\chi_c(t_j-t_i,y-x)\,dy &= \int_{\R^d}\frac{c^d}{(t_j-t_i)^{d/2}}\,Q_{M-(d+5+\gamma)}\!\left(\frac{|z|}{(t_j-t_i)^{1/2}/c}\right)dz\\
&= \int_{\R^d}Q_{M-(d+5+\gamma)}(w)\,dw = 1.
\end{aligned}
\end{equation*}
The same computation gives $\int_{\R^d}\rho_c(t_j-t_i,y-x)\,dy=1$. The result follows immediately.
\end{proof}

Repeating the telescoping decomposition of Lemma~\ref{lem:telescope} for the finite-dimensional densities of the vectors $(Y_{\tau_1},\dots,Y_{\tau_m})$ and $(Y^{(\epsilon)}_{\tau_1},\dots,Y^{(\epsilon)}_{\tau_m})$ (respectively $(X_{\tau_1},\dots,X_{\tau_m})$ and $(X^{\epsilon}_{\tau_1},\dots,X^{\epsilon}_{\tau_m})$), with Lemma~\ref{lem:l1pert} in place of Lemma~\ref{lem:l1}, we obtain
\begin{equation*}
\int_{\R^{m\times d}}|P^m_x(z)-P^m_{\epsilon,x}(z)|\,dz \le C_1(q,T)\,m\,\Delta_{\epsilon,\gamma,q}\quad\text{(diffusions)},\qquad\le C_3(T)\,m\,\Delta_{\epsilon,\gamma,\infty}\quad\text{(chains)},
\end{equation*}
and the maximal coupling construction of the proof of Proposition~\ref{thm:grid} concludes the proof of Proposition~\ref{thm:pert}.

\medskip
\noindent\textbf{Degenerate case (Proposition \ref{thm:deg}).} We introduce the corresponding discrete-time Markov chains defined on a time grid $0=t_0<t_1<\dots<t_N=T$ with step size $h=t_{n+1}-t_n$. Let $Z_n=(X_n,Y_n)$ be the state of the chain. The Markov chain associated with the degenerate system is given by:
\begin{equation*}
\begin{cases}
X_{t_{n+1}} = X_{t_n} + b(X_{t_n},Y_{t_n})h + \sigma(X_{t_n},Y_{t_n})(W_{t_{n+1}}-W_{t_n}), \\
Y_{t_{n+1}} = Y_{t_n} + X_{t_n}h,
\end{cases}
\end{equation*}
and similarly for the perturbed version $Z^{\epsilon}_n=(X^{\epsilon}_n,Y^{\epsilon}_n)$ with coefficients $b_{\epsilon},\sigma_{\epsilon}$.

The single-step stability for the degenerate diffusions is provided by \cite[Theorem~4.3.1]{kozhina2018}, whose assumption (AD3) is, up to constants, our \textbf{AD3}, i.e.\ H\"older continuity with respect to the intrinsic distance $|x-x'|+|y-y'|^{1/3}$ (the article \cite[Theorem~1.1]{kozhina2017} gives the same bound under the stronger Euclidean H\"older condition): under assumptions \textbf{A1}, \textbf{A2}, \textbf{AD3}, for $q\in(4d,\infty]$, there exist constants $C:=C_2(q,T)\ge1$ (see \eqref{eq:Cdeg}) and $c\in(0,1]$ such that for all $0<t_i<t_j\le T$ and all $((x,y),(x',y'))\in(\R^{2d})^2$:
\begin{equation*}
|(p-p^{\epsilon})(t_i,t_j,(x,y),(x',y'))| \le C\overline{\Delta}_{\epsilon,\gamma,q}\,p_{c,K}(t_j-t_i,(x,y),(x',y')),
\end{equation*}
where $p_{c,K}$ denotes the density function
\begin{equation*}
p_{c,K}(u,(x,y),(x',y')) := \frac{c^d\,3^{d/2}}{(2\pi u^2)^d}\exp\left(-c\left(\frac{|x'-x|^2}{4u}+\frac{3|y'-y-(x+x')u/2|^2}{u^3}\right)\right).
\end{equation*}
Since $p_{c,K}(u,(x,y),\cdot)$ is a probability density on $\R^{2d}$ (the density of a Gaussian process, see \cite{kozhina2017}), integration of the above bounds yields the single-step $L^1$-estimate
\begin{equation*}
\int_{\R^{2d}}|p-p^{\epsilon}|(t_i,t_j,(x,y),(x',y'))\,dx'dy' \le C\overline{\Delta}_{\epsilon,\gamma,q}.
\end{equation*}
(The analogous bound for the transition densities of the Euler schemes is stated in \cite[Theorem~4.3.5]{kozhina2018}; its proof there, \cite[Lemmas~4.3.7 and~4.3.8]{kozhina2018}, is carried out for $q=\infty$ only. It is not needed for Proposition~\ref{thm:deg}, which concerns the diffusions themselves.) Since this bound precisely matches the single-step bound used in the non-degenerate case, the rest of the proof is the same as before: expressing the finite-dimensional densities of the vectors $(Z_{t_1},\dots,Z_{t_m})$ and $(Z^{\epsilon}_{t_1},\dots,Z^{\epsilon}_{t_m})$ as products of one-step transition densities, applying the telescoping sum of Lemma~\ref{lem:telescope} over the $m$ time steps, and successively integrating with respect to the spatial variables, we obtain
\begin{equation*}
\|P^m_{(x,y)}-P^m_{\epsilon,(x,y)}\|_{TV} \le C_2(q,T)\,m\,\overline{\Delta}_{\epsilon,\gamma,q}.
\end{equation*}
The maximal coupling construction of the proof of Proposition~\ref{thm:grid} then concludes the proof of Proposition~\ref{thm:deg}.

\subsection{Fluctuations of the chain between the points of the macro-grid: proof of Proposition \ref{thm:chainfluct}}

Recall that the micro-grid is $\Tn=\{t_0,\dots,t_n\}$, $t_i=ih$, $h=\frac Tn$, and the macro-grid is $\Tm=\{\tau_0,\dots,\tau_m\}$, $\tau_k=k\frac Tm$, with $m$ dividing $n$; each macro-cell contains $N:=n/m$ micro-steps, and
\begin{equation*}
Nh=\frac Tm,\qquad \sqrt h\,N=\frac{\sqrt{Tn}}{m}.
\end{equation*}
Let $\F_l:=\sigma(Y_{t_0},\dots,Y_{t_l},\eta_1,\dots,\eta_l)$, so that $Y_{t_l}-Y_{t_{l-1}}=h\,b_{l-1}+\sqrt h\,\zeta_l$ with $\zeta_l=\Sigma_{l-1}\eta_l$, $|b_{l-1}|\le K$, $|\Sigma_{l-1}|\le\kappa$. For the chain \eqref{eq:chain}, $b_l=b(t_l,Y_{t_l})$ and $\Sigma_l=\sigma(t_l,Y_{t_l})$.

\emph{Tail bounds.} Conditionally on $\F_{l-1}$, $\eta_l$ is centered (hence so is $\zeta_l$) with a density bounded by $C_*Q_r$, $r>d+1$; put $\rho=r-d>1$. For $u>0$,
\begin{equation*}
\int_{|z|>u}Q_r(z)\,dz\le c_r\omega_d\int_u^{\infty}s^{d-1-r}\,ds=\frac{c_r\omega_d}{\rho}\,u^{-\rho},\qquad
\int_{|z|>u}|z|\,Q_r(z)\,dz\le\frac{c_r\omega_d}{\rho-1}\,u^{1-\rho}.
\end{equation*}
Since $|\zeta_l|\le\kappa|\eta_l|$, we obtain for every $R>0$, with $c=C_*c_r\omega_d\kappa^{\rho}/(\rho-1)$ as in Proposition~\ref{thm:chainfluct},
\begin{equation}\label{eq:chtails}
\PP\left(|\zeta_l|>R\mid\F_{l-1}\right)\le\PP\left(|\eta_l|>R/\kappa\mid\F_{l-1}\right)\le cR^{-\rho},\qquad
\E\left[|\zeta_l|\,\ind{|\zeta_l|>R}\mid\F_{l-1}\right]\le cR^{1-\rho}.
\end{equation}

\emph{Decomposition.} Fix a cell $[\tau_k,\tau_{k+1}]$ with $\tau_k=t_q$, $\tau_{k+1}=t_{q+N}$, and fix $R>0$. For $q<l\le q+N$ write
\begin{equation*}
Y_{t_l}-Y_{\tau_k}=D_l+B_l+M_l+E_l,
\end{equation*}
where
\begin{align*}
D_l&:=h\sum_{j=q+1}^lb_{j-1}, &
B_l&:=\sqrt h\sum_{j=q+1}^l\E\left[\zeta_j\ind{|\zeta_j|\le R}\mid\F_{j-1}\right],\\
M_l&:=\sqrt h\sum_{j=q+1}^l\left(\zeta_j\ind{|\zeta_j|\le R}-\E\left[\zeta_j\ind{|\zeta_j|\le R}\mid\F_{j-1}\right]\right), &
E_l&:=\sqrt h\sum_{j=q+1}^l\zeta_j\ind{|\zeta_j|>R}.
\end{align*}

\emph{Drift and bias.} Since $|b_{j-1}|\le K$, $|D_l|\le KNh=KT/m$. Since $\E[\zeta_j\mid\F_{j-1}]=0$, we have $\E[\zeta_j\ind{|\zeta_j|\le R}\mid\F_{j-1}]=-\E[\zeta_j\ind{|\zeta_j|>R}\mid\F_{j-1}]$, so by \eqref{eq:chtails}
\begin{equation*}
|B_l|\le\sqrt h\,N\,cR^{1-\rho}=c\sqrt{Tn}\;m^{-1}R^{1-\rho}.
\end{equation*}

\emph{Large errors.} $E_l=0$ for all $l$ on the event $G_k:=\{|\zeta_j|\le R,\ q<j\le q+N\}$, and by \eqref{eq:chtails}
\begin{equation*}
\PP(G_k^c)\le\sum_{j=q+1}^{q+N}\E\left[\PP\left(|\zeta_j|>R\mid\F_{j-1}\right)\right]\le N\,cR^{-\rho}.
\end{equation*}

\emph{Martingale part.} We recall the Azuma--Hoeffding inequality (see Corollary~3.9 in \cite{vanhandel2014}).

\begin{proposition}[Azuma--Hoeffding inequality]\label{prop:azuma}
Let $\{\F_k\}_{k\le n}$ be any filtration, and let $\eta_k,A_k,B_k$ satisfy the following properties for $k=1,2,\dots,n$:
\begin{enumerate}
\item Martingale difference property: $\eta_k$ is $\F_k$-measurable and $\E(\eta_k\mid\F_{k-1})=0$.
\item Predictable bounds: $A_k,B_k$ are $\F_{k-1}$-measurable and $A_k\le\eta_k\le B_k$ a.s.
\end{enumerate}
Then $\sum_{k=1}^n\eta_k$ is sub-Gaussian with variance proxy $\frac14\sum_{k=1}^n|B_k-A_k|^2_{\infty}$. In particular, we obtain for every $t\ge0$ the tail bound
\begin{equation*}
\PP\Bigg[\sum_{k=1}^n\eta_k\ge t\Bigg]\le\exp\left(-\frac{2t^2}{\sum_{k=1}^n|B_k-A_k|^2_{\infty}}\right).
\end{equation*}
\end{proposition}

Fix a unit vector $e_i$ of the canonical basis and let $\eta_j:=\langle M_j-M_{j-1},e_i\rangle=\sqrt h\,\langle\zeta_j\ind{|\zeta_j|\le R}-\E[\zeta_j\ind{|\zeta_j|\le R}\mid\F_{j-1}],e_i\rangle$ ($M_q:=0$). This is a martingale difference sequence, and conditionally on $\F_{j-1}$ it takes values in an interval $[A_j,B_j]$ of length $B_j-A_j=2R\sqrt h$ with $\F_{j-1}$-measurable endpoints. By Proposition~\ref{prop:azuma}, for every $\lambda\in\R$,
\begin{equation*}
\E\exp\Big(\lambda\sum_{j=q+1}^{q+N}\eta_j\Big)\le\exp\left(\frac{\lambda^2}{2}\cdot\frac14\cdot N\,(2R\sqrt h)^2\right)=\exp\left(\frac{\lambda^2R^2T}{2m}\right).
\end{equation*}
Since $l\mapsto\exp(\lambda\langle M_l,e_i\rangle)$ is a nonnegative submartingale, Doob's maximal inequality gives, for $x>0$ and $\lambda>0$,
\begin{equation*}
\begin{aligned}
\PP\Big(\max_{q<l\le q+N}\langle M_l,e_i\rangle\ge x\Big)&\le e^{-\lambda x}\,\E\exp\big(\lambda\langle M_{q+N},e_i\rangle\big)\\
&\le\exp\left(-\lambda x+\frac{\lambda^2R^2T}{2m}\right)=\exp\left(-\frac{x^2m}{2R^2T}\right)
\end{aligned}
\end{equation*}
for $\lambda=xm/(R^2T)$. The same bound holds for $-\langle M_l,e_i\rangle$. If $|M_l|\ge x$, then $|\langle M_l,e_i\rangle|\ge x/\sqrt d$ for some $i$; hence, with $x=m^{\alpha-1}$,
\begin{equation*}
\PP\Big(\max_{q<l\le q+N}|M_l|\ge m^{\alpha-1}\Big)\le2d\exp\left(-\frac{m^{2\alpha-1}}{2dR^2T}\right).
\end{equation*}

\emph{Conclusion.} On $G_k\cap\{\max_l|M_l|<m^{\alpha-1}\}$ we have
\begin{equation*}
\max_{q<l\le q+N}|Y_{t_l}-Y_{\tau_k}|<m^{\alpha-1}+KTm^{-1}+c\sqrt{Tn}\;m^{-1}R^{1-\rho}.
\end{equation*}
For $t\in[\tau_k,\tau_{k+1}]$ the linear interpolation $Y_t$ is a convex combination of two neighbouring grid values in the same cell, so $|Y_t-Y_{\tau_k}|\le\max_{q<l\le q+N}|Y_{t_l}-Y_{\tau_k}|$. Therefore
\begin{equation*}
\PP\Big\{\sup_{t\in[\tau_k,\tau_{k+1}]}|Y_t-Y_{\tau_k}|\ge m^{\alpha-1}+KTm^{-1}+c\sqrt{Tn}\;m^{-1}R^{1-\rho}\Big\}\le2d\exp\left(-\frac{m^{2\alpha-1}}{2dR^2T}\right)+N\,cR^{-\rho},
\end{equation*}
and summing over the $m$ cells (with $mN=n$) gives Proposition~\ref{thm:chainfluct}. This completes the proof of Proposition~\ref{thm:chainfluct}. \qed

\begin{remark}[on the scale of the fluctuations]\label{rem:scale}
The threshold of Proposition~\ref{thm:chainfluct} is of the same order $m^{\alpha-1}$ as the threshold $m^{\alpha-1}+KTm^{-1}$ for the diffusion, and it cannot be taken of the order of the cell length $T/m$: under \textbf{A2} the martingale part of $Y_{\tau_{k+1}}-Y_{\tau_k}$ has covariance of order $Nh=T/m$, i.e.\ standard deviation of order $\sqrt{T/m}\gg T/m$ for $m\gg T$, and by the central limit theorem the probability that it exceeds $c\,T/m$ does not tend to zero.
\end{remark}

\subsection{Uniform closeness of the perturbed diffusions: proof of Proposition \ref{thm:supdiff}}

The proof follows the scheme of Step~1 of the proof of Proposition~\ref{thm:sup}. We consider as a continuous local martingale the process
\begin{equation*}
M_{s,i}:=\int_{\tau_i}^{\tau_i+s}\sigma(t,X_t)\,\de W_t,\qquad 0\le s\le\tfrac Tm .
\end{equation*}
According to \textbf{(A1)} the diffusion matrices (original and perturbed) are bounded by $K$, so each coordinate $M^{(l)}_{s,i}$ is a continuous martingale. Now we use the following well-known fact (see \cite{bass2011}, p.~86, Problem 12.9).

\begin{lemma}\label{lem:bass}
Let $M_t$ be a continuous local martingale. The following inequality holds
\begin{equation*}
\PP\left(\sup_{t\ge0}|M_t|>x,\ \langle M\rangle_{\infty}<y\right)\le2e^{-\frac{x^2}{2y}}.
\end{equation*}
\end{lemma}

By \textbf{(A1)}, $\langle M^{(l)}_{\cdot,i}\rangle_{T/m}\le K^2T/m$. Applying Lemma~\ref{lem:bass} to $M^{(l)}_{\cdot\wedge T/m,i}$ with $y=y'>K^2T/m$ and letting $y'\downarrow K^2T/m$,
\begin{equation*}
\PP\Big(\sup_{0\le s\le\frac{T}{m}}|M^{(l)}_{s,i}|>x\Big)\le2e^{-\frac{x^2m}{2K^2T}}.
\end{equation*}
If $|M_{s,i}|>m^{\alpha-1}$, then $|M^{(l)}_{s,i}|>m^{\alpha-1}/\sqrt d$ for some $l$; with $x=m^{\alpha-1}/\sqrt d$ and since $\int_{\tau_i}^{\tau_{i+1}}|b(t,X_t)|\,\de t\le KTm^{-1}$,
\begin{equation*}
\PP\Big(\sup_{0\le s\le\frac{T}{m}}|\Delta X_{s,i}|>m^{\alpha-1}+\frac{KT}{m}\Big)\le\PP\Big(\sup_{0\le s\le\frac{T}{m}}|M_{s,i}|>m^{\alpha-1}\Big)\le2d\,e^{-\frac{m^{2\alpha-1}}{2dK^2T}}.
\end{equation*}
The union bound over the $m$ cells gives
\begin{equation*}
\PP\Big(\sup_{0\le i\le m-1}\sup_{0\le s\le\frac{T}{m}}|\Delta X_{s,i}|>m^{\alpha-1}+\frac{KT}{m}\Big)\le2dm\,e^{-\frac{m^{2\alpha-1}}{2dK^2T}},
\end{equation*}
and, analogously (the perturbed coefficients are bounded by $K_2\le K$), the same upper bound holds for the perturbed diffusion:
\begin{equation*}
\PP\Big(\sup_{0\le i\le m-1}\sup_{0\le s\le\frac{T}{m}}|\Delta X^{\epsilon}_{s,i}|>m^{\alpha-1}+\frac{KT}{m}\Big)\le2dm\,e^{-\frac{m^{2\alpha-1}}{2dK^2T}}.
\end{equation*}
This proves Proposition~\ref{thm:supdiff}. The same proof applies verbatim to the first component $\widetilde X$ (resp.\ $\widetilde X^{\epsilon}$) of the degenerate diffusion \eqref{eq:degsde} (resp.\ \eqref{eq:degsdepert}), whose dynamics $\de\widetilde X_t=b(\widetilde X_t,\widetilde Y_t)\,\de t+\sigma(\widetilde X_t,\widetilde Y_t)\,\de W_t$ has coefficients bounded by $K$ by \textbf{A1}; this proves Proposition~\ref{thm:supdegdif}.

Clearly, every $s\in[0,T]$ lies in some cell $[\tau_i,\tau_{i+1}]$. Put $A=\left\{\omega:\,X_{\tau_k}=X^{\epsilon}_{\tau_k},\ k=1,\dots,m\right\}$. For the versions constructed in Proposition~\ref{thm:pert} (completed to paths as in Subsection~\ref{sec:results4}) we have, on $A$, $|X_s-X^{\epsilon}_s|\le|X_s-X_{\tau_i}|+|X^{\epsilon}_s-X^{\epsilon}_{\tau_i}|$, and therefore
\begin{equation}\label{eq:supdiffbound}
\PP\Big(\sup_{s\in[0,T]}|X_s-X^{\epsilon}_s|>2m^{\alpha-1}+\frac{2KT}{m}\Big)\le4dm\,e^{-\frac{m^{2\alpha-1}}{2dK^2T}}+C_1(q,T)\,m\,\Delta_{\epsilon,\gamma,q},
\end{equation}
which is Theorem~\ref{thm:dif}. \qed

\begin{remark}\label{rem:largeT}
The case $T\nearrow\infty$ may be treated with the explicit constants above. Take in Theorem~\ref{thm:dif} the grid $\tau_i=\frac{iT}{m(T)}$ with $m(T)=\lfloor T^{\theta}\rfloor$, $\theta>\frac{1}{2\alpha-1}$ (so that $\theta>1$). For $T\ge2^{1/\theta}$ we have $m(T)\ge T^{\theta}/2$, hence the first summand in \eqref{eq:supdiffbound} is at most $4dT^{\theta}\exp\big(-\frac{(T^{\theta}/2)^{2\alpha-1}}{2dK^2T}\big)\to0$, and the threshold $2m(T)^{\alpha-1}+2KT/m(T)\le2m(T)^{\alpha-1}+4KT^{1-\theta}\to0$. The constant of Proposition~\ref{thm:pert} depends on $T$: by \eqref{eq:CKKM} and the asymptotics $\mathcal{E}_{\vartheta_q,1}(z)\sim\vartheta_q^{-1}\exp(z^{1/\vartheta_q})$, $z\to\infty$, it grows at most like $\exp\big(cT^{1+a_q/\vartheta_q}\big)$ for some $c>0$ independent of $T$. The coupling term $C_1(q,T)\,m(T)\,\Delta_{\epsilon,\gamma,q}$ in \eqref{eq:supdiffbound} is therefore small only if $\Delta_{\epsilon,\gamma,q}\ll T^{-\theta}\exp\big(-cT^{1+a_q/\vartheta_q}\big)$.
\end{remark}

\nocite{asmussen2001,konakovmammenhuang2026,denhollander2012}

\bibliographystyle{plain}
\bibliography{Refs2}

@book{bhattacharya1976,
  author    = {Bhattacharya, R. N. and Ranga Rao, R.},
  title     = {Normal Approximation and Asymptotic Expansions},
  publisher = {Wiley},
  address   = {New York},
  year      = {1976}
}

@article{bitter2025,
  author  = {Bitter, I. I. and Konakov, V. D.},
  title   = {An asymptotic version of the parametrix method for {M}arkov chains converging to diffusions},
  journal = {Math. Notes},
  volume  = {120},
  number  = {4},
  pages   = {599--615},
  year    = {2026},
  doi     = {10.1134/S0001434626603990}
}

@article{bitter2021,
  author  = {Bitter, I. and Konakov, V.},
  title   = {{$L_1$} and {$L_\infty$} stability of transition densities of perturbed diffusions},
  journal = {Random Oper. Stoch. Equ.},
  volume  = {29},
  number  = {4},
  pages   = {287--308},
  year    = {2021},
  doi     = {10.1515/rose-2021-2067}
}

@article{konakov2017stability,
  author  = {Konakov, V. and Kozhina, A. and Menozzi, S.},
  title   = {Stability of densities for perturbed diffusions and {M}arkov chains},
  journal = {ESAIM Probab. Stat.},
  volume  = {21},
  pages   = {88--112},
  year    = {2017},
  doi     = {10.1051/ps/2016028}
}

@article{kozhina2017,
  author  = {Kozhina, A.},
  title   = {Stability of densities for perturbed degenerate diffusions},
  journal = {Theory Probab. Appl.},
  volume  = {61},
  number  = {3},
  pages   = {489--499},
  year    = {2017},
  doi     = {10.1137/S0040585X97T988290}
}

@phdthesis{kozhina2018,
  author = {Kozhina, A.},
  title  = {Parametrix Method and its Applications in Probability Theory},
  school = {Ruprecht-Karls-Universit{\"a}t Heidelberg},
  year   = {2018},
  doi    = {10.11588/heidok.00025129}
}

@article{konakovmammen2000,
  author  = {Konakov, V. and Mammen, E.},
  title   = {Local limit theorems for transition densities of {M}arkov chains converging to diffusions},
  journal = {Probab. Theory Related Fields},
  volume  = {117},
  pages   = {551--587},
  year    = {2000}
}

@article{konakovmammen2002,
  author  = {Konakov, V. and Mammen, E.},
  title   = {Edgeworth type expansions for {E}uler schemes for stochastic differential equations},
  journal = {Monte Carlo Methods Appl.},
  volume  = {8},
  pages   = {271--285},
  year    = {2002}
}

@article{konakovmenozzi2017,
  author  = {Konakov, V. and Menozzi, S.},
  title   = {Weak error for the {E}uler scheme approximation of diffusions with non-smooth coefficients},
  journal = {Electron. J. Probab.},
  volume  = {22},
  year    = {2017},
  doi     = {10.1214/17-EJP53}
}

@article{konakovmenozzimolchanov2010,
  author  = {Konakov, V. and Menozzi, S. and Molchanov, S.},
  title   = {Explicit parametrix and local limit theorems for some degenerate diffusion processes},
  journal = {Ann. Inst. Henri Poincar{\'e} Probab. Stat.},
  volume  = {46},
  pages   = {908--923},
  year    = {2010}
}

@article{menozzi2011,
  author  = {Menozzi, S.},
  title   = {Parametrix techniques and martingale problems for some degenerate {K}olmogorov equations},
  journal = {Electron. Commun. Probab.},
  volume  = {16},
  pages   = {234--250},
  year    = {2011}
}

@article{menozzipescezhang2021,
  author  = {Menozzi, S. and Pesce, A. and Zhang, X.},
  title   = {Density and gradient estimates for non degenerate {B}rownian {SDE}s with unbounded measurable drift},
  journal = {J. Differential Equations},
  volume  = {272},
  pages   = {330--369},
  year    = {2021},
  doi     = {10.1016/j.jde.2020.09.004}
}

@incollection{bassperkins2009,
  author    = {Bass, R. F. and Perkins, E. A.},
  title     = {A new technique for proving uniqueness for martingale problems},
  booktitle = {From Probability to Geometry ({I}): Volume in Honor of the 60th Birthday of {J}ean-{M}ichel {B}ismut},
  series    = {Ast{\'e}risque},
  publisher = {Soci{\'e}t{\'e} Math{\'e}matique de France},
  volume    = {327},
  pages     = {47--53},
  year      = {2009}
}

@book{bass2011,
  author    = {Bass, R. F.},
  title     = {Stochastic Processes},
  series    = {Cambridge Series in Statistical and Probabilistic Mathematics},
  volume    = {33},
  publisher = {Cambridge University Press},
  address   = {Cambridge},
  year      = {2011}
}

@article{bally1996,
  author  = {Bally, V. and Talay, D.},
  title   = {The law of the {E}uler scheme for stochastic differential equations. {II}. {C}onvergence rate of the density},
  journal = {Monte Carlo Methods Appl.},
  volume  = {2},
  pages   = {93--128},
  year    = {1996}
}

@article{mikulevicius1991,
  author  = {Mikulevi{\v{c}}ius, R. and Platen, E.},
  title   = {Rate of convergence of the {E}uler approximation for diffusion processes},
  journal = {Math. Nachr.},
  volume  = {151},
  pages   = {233--239},
  year    = {1991}
}

@article{genoncatalot1993,
  author  = {Genon-Catalot, V. and Jacod, J.},
  title   = {On the estimation of the diffusion coefficient for multi-dimensional diffusion processes},
  journal = {Ann. Inst. Henri Poincar{\'e} Probab. Stat.},
  volume  = {29},
  pages   = {119--151},
  year    = {1993}
}

@article{corielli2010,
  author  = {Corielli, F. and Foschi, P. and Pascucci, A.},
  title   = {Parametrix approximation of diffusion transition densities},
  journal = {SIAM J. Financial Math.},
  volume  = {1},
  pages   = {833--867},
  year    = {2010}
}

@article{benhamou2010,
  author  = {Benhamou, E. and Gobet, E. and Miri, M.},
  title   = {Expansion formulas for {E}uropean options in a local volatility model},
  journal = {Int. J. Theor. Appl. Finance},
  volume  = {13},
  pages   = {603--634},
  year    = {2010}
}

@article{bogachev2016,
  author  = {Bogachev, V. I. and R{\"o}ckner, M. and Shaposhnikov, S. V.},
  title   = {Distances between transition probabilities of diffusions and applications to nonlinear {F}okker--{P}lanck--{K}olmogorov equations},
  journal = {J. Funct. Anal.},
  volume  = {271},
  number  = {5},
  pages   = {1262--1300},
  year    = {2016},
  doi     = {10.1016/j.jfa.2016.05.016}
}

@article{galeati2023,
  author  = {Galeati, L. and Ling, C.},
  title   = {Stability estimates for singular {SDE}s and applications},
  journal = {Electron. J. Probab.},
  volume  = {28},
  number  = {24},
  pages   = {1--31},
  year    = {2023}
}

@misc{denhollander2012,
  author       = {den Hollander, F.},
  title        = {Probability Theory: The Coupling Method},
  howpublished = {Lecture notes, Leiden University},
  year         = {2012}
}

@misc{vanhandel2014,
  author       = {van Handel, R.},
  title        = {Probability in High Dimension},
  howpublished = {Lecture notes (APC 550), Princeton University},
  year         = {2014}
}

@book{revuzyor1999,
  author    = {Revuz, D. and Yor, M.},
  title     = {Continuous Martingales and {B}rownian Motion},
  edition   = {3rd},
  series    = {Grundlehren der mathematischen Wissenschaften},
  volume    = {293},
  publisher = {Springer},
  address   = {Berlin},
  year      = {1999},
  doi       = {10.1007/978-3-662-06400-9}
}

@misc{bally2024,
  author = {Bally, V. and Qin, Y.},
  title  = {Total variation distance between a jump-equation and its {G}aussian approximation},
  note   = {Preprint, arXiv:2212.07417. TODO: check journal version},
  year   = {2022}
}

@misc{hamaguchi2024,
  author = {Hamaguchi, Y.},
  title  = {TODO: title},
  note   = {TODO: complete this reference},
  year   = {2024}
}

@misc{babi2022,
  author = {Babi},
  title  = {TODO: title},
  note   = {TODO: complete this reference},
  year   = {2022}
}

@article{asmussen2001,
  author  = {Asmussen, S. and Rosi{\'n}ski, J.},
  title   = {Approximations of small jumps of {L}{\'e}vy processes with a view towards simulation},
  journal = {J. Appl. Probab.},
  volume  = {38},
  number  = {2},
  pages   = {482--493},
  year    = {2001}
}

@misc{konakovmammenhuang2026,
  author = {Konakov, V. and Mammen, E. and Huang, L.},
  title  = {TODO: title},
  note   = {TODO: complete this reference},
  year   = {2026}
}

\end{document}